%% file: Journal.tex
\documentclass[11pt]{amsart}

\usepackage{a4wide}
\usepackage{amsfonts,amsmath}

\usepackage[usenames]{color}
\usepackage[dvips]{graphicx}

\newdimen\margin   
\def\COMMENT#1{}
\let\COMMENT=\footnote

\newenvironment{remark}{\noindent {\bf Remark}.}{\par\smallskip\par}

\newcommand{\eps}{\varepsilon}
\newcommand{\al}{\alpha}

\newcommand{\prob}{\mathbb{P}}
\newcommand{\ex}{\mathbb{E}}
\newcommand{\Pa}{{\mathcal P}}

\newcommand{\E}{{\mathcal E}}

\newcommand{\C}{{\mathcal C}}

\newcommand{\X}{{\mathcal X}}

\newcommand{\A}{{\mathcal A}}

\newcommand{\T}{{\mathcal T}}
\newcommand{\G}{{\mathcal{G}}}

\newcommand{\N}{{\mathcal N}}

\newcommand{\Exx}{\mathrm{Ex}}

\newcommand{\Nl}{\mathsf{N}}
\newcommand{\Cl}{\mathsf{C}}
\newcommand{\Bl}{\mathsf{B}}

\newcommand{\Y}{{\mathcal{Y}}}
\newcommand{\BP}{{\mathbb{P}}}

\newcommand{\CN}{{\mathcal{N}}}
\newcommand{\CS}{{\mathcal{S}}}

\newcommand{\CP}{{\mathcal{P}}}
\newcommand{\CH}{{\mathcal{H}}}
\newcommand{\CT}{{\mathcal{T}}}
\newcommand{\Net}{{\mathtt{Net}}}

\newcommand{\Be}{{\mathtt{Be}}}
\newcommand{\Ser}{{\mathtt{Ser}}}
\newcommand{\Par}{{\mathtt{Par}}}
\newcommand{\sh}{{\mathtt{sh}}}
\newcommand{\Set}{{\mathsf{Set}}}

\renewcommand{\Pr}[1]{\mathbb{P}\left(#1\right)}
\renewcommand{\E}[1]{\mathbb{E}\left(#1\right)}

\newcommand{\keyword}[1]{\textbf{#1}~}
\newcommand{\IF}{\keyword{if}}

\newcommand{\THEN}{\keyword{then} }
\newcommand{\ELSE}{\keyword{else}}
\newcommand{\RETURN}{\keyword{return}}
\newcommand{\FOR}{\keyword{for}}
\newcommand{\FOREACH}{\keyword{foreach}}

\newtheorem{firsttheorem}{Proposition}

\newtheorem{theorem}[firsttheorem]{Theorem}
\newtheorem{lemma}[firsttheorem]{Lemma}
\newtheorem{corollary}[firsttheorem]{Corollary}

\newtheorem{definition}[firsttheorem]{Definition}
\newtheorem{proposition}[firsttheorem]{Proposition}

\numberwithin{equation}{section}
\numberwithin{firsttheorem}{section}

\begin{document}
\title{3-Connected Cores In Random Planar Graphs}

\author{Nikolaos Fountoulakis and Konstantinos Panagiotou}
\maketitle
\vspace{-0.75cm}
\begin{center}
 \small Max-Planck-Institute for Informatics \\
Saarbr\"ucken, Germany
\end{center}

\begin{abstract}
The study of the structural properties of large random planar graphs has become in recent years a field of intense research in computer science and discrete mathematics. Nowadays, a random planar graph is an important and challenging model for evaluating methods that are developed to study properties of random graphs from classes with structural side constraints. 

In this paper we focus on the structure of random biconnected planar graphs regarding the sizes of their 3-connected building blocks, which 
we call \emph{cores}. In fact, we prove a general theorem regarding random biconnected graphs from various classes.
If $\Bl_n$ is a graph drawn uniformly at random from a class $\mathcal{B}$ of labeled biconnected graphs, then we show 
that with probability $1-o(1)$ as $n \rightarrow \infty$, $\Bl_n$ belongs to exactly one of the following categories:
\begin{itemize}
	\item[(i)] Either there is a unique \emph{giant} core in $\Bl_n$, that is, there is a $0<c = c(\mathcal{B})<1$ such that the largest core contains $\sim cn$ vertices, and every other core contains at most $n^{\alpha}$ vertices, where $0 < \alpha = \alpha({\mathcal B}) < 1$;
	\item[(ii)] or all cores of $\Bl_n$ contain $O(\log n)$ vertices.
\end{itemize}
Moreover, we find the critical condition that determines the category to which $\Bl_n$ belongs, and also provide sharp concentration results for the counts of cores of all sizes between 1 and~$n$. As a corollary, we obtain that a random biconnected planar graph belongs to category (i), where in particular $c = 0.765\dots$ and $\alpha = 2/3$.
\end{abstract}

\section{Introduction}

A fundamental discipline of computer science is the theoretical and practical evaluation of the performance of algorithms. From a theoretical viewpoint, many important computational problems turn out to be $\mathcal{NP}$-hard or even very difficult to approximate, which leaves little hope for finding efficient algorithms that solve all instances. On the other hand, in practice it is widely observed that one can find satisfactory solutions efficiently, even when using algorithms that are known to perform very badly on certain inputs. In other words, it seems that ``typical'' instances are in some sense easier and the classical worst-case analysis may be too pessimistic. This motivates the study of the performance of algorithms from an \emph{average} point of view.  However, in order to perform such an analysis, it is necessary to specify an appropriate \emph{probability distribution} on the set of all inputs.

After having a suitable model at hand, an important ingredient of a meaningful average-case analysis is the precise knowledge of the typical structure of an input sampled from the specified distribution. For instance, in the context of graph problems, one wants to study the properties of graphs in the corresponding random graph model. However, this may be a very challenging task: if the space of inputs consists of graphs with global structural constraints, then the existence of dependencies can make the analysis formidable.

A natural graph class with such constraints that has attracted the attention of researchers in computer science and discrete mathematics in recent years is the class of \emph{random planar graphs}. In fact, since it was first studied by Denise, Vasconcellos, and Welsh~\cite{DVW}, this model has evolved as the primary example of studying random graphs from constrained classes, see e.g.\ \cite{MSW,GimNoy1}. From the perspective of the average-case analysis, it would be helpful to know that although a random planar graph inherits the dependencies that arise from the requirement of planarity, essentially most of its vertices behave independently of each other. However, the planarity condition makes almost all the tools and methods that have been used in the past decades for the analysis of classical random graphs models fail in the case of such a problem. Consequently, the development of new approaches is necessary and essential. 

One attempt to resolve this issue was taken recently by the second author and Steger~\cite{PanSteg}. The precise setting they considered is as follows. Let $\mathcal{C}$ be a class of labeled connected graphs, and let $\Cl_n$ be a random graph from $\mathcal{C}$ with $n$ vertices. The main idea in their work is to consider the \emph{maximal biconnected components} of a connected graph $\Cl_n$. In this context, they showed, among other results, that under certain assumptions $\Cl_n$ belongs a.a.s.\footnote{asymptotically almost surely, i.e.\ with probability tending to 1 when $n\to\infty$} to exactly one of the following categories:
\begin{itemize}
	\item[(i)] There is a unique \emph{giant} biconnected component in $\Cl_n$. More precisely, there is a $0<c = c(\mathcal{C})<1$ such that the largest biconnected component contains $\sim cn$ vertices, while every other component contains $o(n)$ vertices.
	\item[(ii)] All biconnected components contain $O(\log n)$ vertices.
\end{itemize}
Additionally, in \cite{PanSteg} it was shown that random planar graphs belong to the former category, whereas e.g.\ random outerplanar graphs belong to the latter. Observe that for graphs that belong to category (ii), almost all pairs of vertices lie in different biconnected components, while this is not the case for graphs from the first category. A consequence of these facts is the following important observation. Random graphs from classes that belong to category (ii) ``contain'' in a well-defined sense plenty of independence. In particular, any such graph can be generated by choosing independently every one of its biconnected components, and gluing them together at the cut-vertices. As the biconnected components contain few vertices, and as they intersect each other only at single vertices, the \emph{impact} of each block to the whole graph is small. So, such graphs resemble in a certain way the behavior of classical random graphs, where each edge is included independently with a specified probability, with the difference that here we choose the blocks independently of each other. However, random graphs from classes that belong to category~(i) do not have this property: a lot of structure that we cannot control is ``hidden'' in the giant biconnected component, which contains a constant fraction of the vertices.

This motivates a finer analysis regarding the typical structure of random biconnected graphs, which is the main topic of this work. In particular, we investigate how and under which conditions we can decompose a biconnected graph into building blocks of higher complexity. As we shall see shortly, such a decomposition is well-known and possible. We will show that we encounter again a fundamental \emph{dichotomy}: depending on some critical condition, which we determine explicitly, we prove that large biconnected graphs have either only ``small'' building blocks of higher complexity, or a constant fraction of the vertices is contained in such a building block. Hence, we discover a picture that is completely analogous to the distribution of the sizes of the blocks in random connected graphs.

Concerning the actual analysis, our methods not only differ, but also extend significantly some of the ideas presented in~\cite{PanSteg} with respect to two major issues. Firstly, essential ingredients of our proofs are precise asymptotic estimates for the number of biconnected graphs in the class in question. Here we resolve this issue in an analytic context and obtain very general and precise enumeration results that are applicable in a wide variety of scenarios frequently encountered in modern theories of asymptotic enumeration. Secondly, regarding the combinatorial aspect, the further decomposition of biconnected graphs in building blocks of higher complexity is considerably more involved than the decomposition of connected graphs into their biconnected components. So, although we use the same basic idea as in~\cite{PanSteg}, namely to sample algorithmically a graph from the class in question, the nature of the resulting algorithms is significantly more complex and cannot be analyzed with standard methods from probability theory. We resolve this issue by proposing a powerful method that relates by equation systems the simultaneous evolution of \emph{several} random variables during the sampling process. We believe that this analysis opens up the possibility to study parameters of~\emph{arbitrarily} complex classes that can be described within the framework of the \emph{symbolic method}, which nowadays is a widely used tool to decompose various classes of combinatorial objects.

\subsection*{Our results}
The main idea of graph classification and enumeration that we will exploit goes back to Tutte~\cite{Tutte} and consists of describing graphs in terms of building blocks of higher \emph{connectivity}. For example, an arbitrary graph can be described by its connected components. Similarly, any connected graph decomposes uniquely into the set of its maximal biconnected components. For the class of biconnected graphs there is a similar decomposition with respect to 3-connected building blocks. Informally, 
one can give a recursive description of biconnected graphs in which 3-connected graphs have their edges replaced by biconnected graphs. This decomposition has been formalized through the notion of a \emph{network} whose precise definition will be given in Section~\ref{Sec:Nets}. The 3-connected building blocks which play the above role are called the \emph{cores} of a biconnected graph.  Using this decomposition, we can define classes of biconnected graphs by restricting the cores to be within a certain class of 3-connected graphs. For example, if we use the class of all 3-connected planar graphs as the base of the decomposition, then we obtain the classes consisting of all biconnected, connected, and general planar graphs. 

In this paper we study the distribution of the sizes of the cores in large random biconnected graphs, that are specified in Tutte's sense by their 3-connected building blocks. Our first result, in a very weakened form, is about random planar graphs. Here and in the remainder we write ``$x \pm y$'' for the interval $(x-y, x+y)$.
\begin{theorem} \label{thm:planar}
Let $\Bl_n$ denote a random labeled biconnected planar graph with $n$ vertices, and let $\omega = \omega(n)$ be any function such that $\lim_{n\to\infty}\omega(n)=\infty$. Let $\eps >0$. Then, $\Bl_n$ contains a.a.s.\ a \emph{unique} largest core with $(1\pm \eps)cn$ vertices, where $c=0.765\ldots$ is given explicitely. Moreover, every other core contains at most $n^{2/3}\omega(n)$ vertices, and there are $(1\pm\eps)c_\ell n$ cores with $\ell$ vertices, where $4 \le \ell = O(n^{2/3})$ and the $c_\ell$ are given explicitely.
\end{theorem}
The main results of this paper are far more general. In fact, we show that a random planar graph is  just a special case of a universal phenomenon. Let $\Bl_n$ be a graph drawn uniformly at random from a class $\mathcal{B}$ of labeled biconnected graphs. In Theorems~\ref{Main:Subcritical} and~\ref{Main:Critical} we determine a critical condition according to which $\Bl_n$ belongs a.a.s.\ to exactly one of those categories:
\begin{itemize}
	\item[(i)] There is a unique \emph{giant} core in $\Bl_n$: there are $0<  c=c(\mathcal{B}), \alpha = \alpha(\mathcal{B}) <1$ such that the largest core contains $(1\pm \eps)cn$ vertices, while every other core contains $O(n^{\alpha})$ vertices.
	\item[(ii)] All cores of $\Bl_n$ contain $O(\log n)$ vertices.
\end{itemize}
Moreover, in each of the above cases we determine for all $4 \le \ell\le n$ the expected number of cores with $\ell$ vertices, and provide very sharp accompanying large deviation results. We refer the reader to Section~\ref{Sec:Nets} for the formal presentation of our setting, and to Theorems~\ref{Main:Subcritical} and~\ref{Main:Critical} for a precise formulation of the above statements.

Our results have implications for several classes of biconnected graphs different from planar graphs. Let us mention selectively two examples. Writing $\Exx (G_1,\ldots, G_k)$ for the class of biconnected graphs that
do not contain $G_1,\ldots, G_k$ as minors, the main result in~\cite{GimNoy} together with our result imply 
that a random graph from $\Exx (K_{3,3})$ on $n$ vertices a.a.s.\ exhibits a behavior as above: it has a unique giant core whereas 
every other core is of lower order. Also, for the class $\Exx (K_5-e)$, which consists of all biconnected graphs whose cores belong to the set 
$\{K_{3,3}, K_3\times K_2\}\cup \{ W_n\}_{n\geq 3}$, where $W_n$ denotes a wheel on $n+1$ vertices, again the main result in~\cite{GimNoy} together with our results imply that a random graph on $n$ vertices from this class contains a.a.s.\ only cores having at most $O(\log n)$ vertices.

\subsection*{Outline}
As we mentioned above, the decomposition of a class of biconnected graphs is facilitated by the notion of a network, which we formally describe in the next section. This is a recursive decomposition that will allow us in Section~\ref{Sec:Samplers} to design efficient algorithms that sample networks according to the so-called \emph{Boltzmann model}. Such ideas were used for the first time e.g. in~\cite{BerPanSte}, where the number of vertices of a given degree in random dissections of convex polygons was studied.
However, the nature of the decomposition in this work is very complex, and this makes the analysis of the samplers significantly more involved compared to the analysis in~\cite{PanSteg,BerPanSte}. In particular, in order to sample networks, we use four different randomized algorithms, which call each other recursively in a nested and unpredictable fashion. This difficulty requires a completely new treatment as far as the analysis is concerned, and the details are presented in Section~\ref{Sec:Samplers}.

Another crucial ingredient in our proof is a precise counting estimate for the number of networks, parametrized by the number of vertices and edges. To this end, we extend and further develop in Section~\ref{sec:singAnalysis} certain known analytic tools, in order to make them applicable in our actual setting. Furthermore, they allow us to show a central limit theorem as well as tail bounds for the number of edges of a (general) random network with $n$ vertices. Finally, in Section~\ref{Sec:Proofs} we give the proofs of the main theorems and in Section~\ref{Sec:remProofs} we give the proofs of a number of auxiliary results which are used in our arguments.

\subsection*{Notational preliminaries}

At this point let us introduce some necessary notation.  
Let~$\C$ be a class of labeled graphs. For any positive integers~$n,m \geq 1$, 
we denote by~$\C_{n,m}$ the subclass of~$\C$ consisting of those graphs that have~$n$ labeled vertices and $m$ edges, and we set $\C_{n} = \cup_{m \ge 0} \C_{n,m}$.
Moreover, we denote by~$C(x,y)$ the \emph{exponential generating function (egf)} for~$\C$, where 
$x$ marks the vertices and~$y$ marks the edges, i.e., the coefficient $[x^ny^m]C(x,y)$ equals $\frac{|C_{n,m}|}{n!}$. 
For a given~$y$, we denote by~$\rho_{C}(y)$ the singularity of~$C(x,y)$ with respect to~$x$ and, in particular, when~$y=1$ 
we write~$\rho_C = \rho_C(1)$. Also, we write~$C_0(y):= C(\rho_C(y),y)$ and~$C_0 := C(\rho_C,1)$. 

We now introduce a number of operations between combinatorial classes that are at the core of our analysis. 
For two graph classes~$\X$ and~$\Y$ we denote by~$\X \times \Y$ the \textit{cartesian product} of~$\X$ and~$\Y$ followed by a relabeling step, so as to guarantee that all labels are distinct. Note that the relation ``$\A = \X \times \Y$'' expresses the fact that there is a bijection between the elements of~$\A$ and pairs of elements from~$\X$ and~$\Y$, but it does not provide any information about how this bijection looks like, i.e., how to construct a graph in~$\A$ from two graphs in~$\X$ and~$\Y$. The same is true for the operators described in the remainder. The class $\X + \Y$ consists of graphs that are either in $\X$ or in $\Y$. We denote by~$\Set_{\ge k}(\X)$ the graph class such that each object in it is an unordered collection of at least $k $ graphs in~$\X$. Finally, the class~$\X \circ_e \Y$ consists of all graphs that are obtained from graphs from~$\X$, where each edge is replaced by a graph from~$\Y$. Here we will usually assume that $\Y$ is a class of graphs with two distinguished vertices, which simply means that we attach a graph $G_e \in \Y$ at each edge $e\in E(G)$, where $G\in \X$, by identifying the special vertices of $G_e$ with the endpoints of $e$. 
This set of combinatorial operators (cartesian product, disjoint union, set, and substitution) appears frequently in modern theories of combinatorial enumeration, and it is beyond the scope of this work to survey them. For a very detailed description and numerous applications we refer to \cite{FlajSed} and references therein.

\section{Networks} \label{Sec:Nets}

\newcommand{\CB}{\mathcal{B}}
\newcommand{\CZ}{\mathcal{Z}}

\subsection*{Biconnected Graphs \& Networks}
Let~$\CB$ be a class of labeled biconnected graphs.
In the remainder we assume that the graph consisting of a single edge is in $\CB$. Before we study the typical properties of a graph drawn
uniformly at random from~$\CB_n$, that is, the subclass of $\CB$ consisting of graphs on $n$ vertices, 
let us introduce an auxiliary graph class that plays an important role in the decomposition of biconnected graphs. 
Following Trakhtenbrot~\cite{Traht} and Tutte~\cite{Tutte} we define a \emph{network} as a connected graph with two ``special'' vertices, called the \emph{left pole} and the \emph{right pole}, such that adding the edge between the poles the resulting (multi-)graph is biconnected. 
The remaining vertices of a network are called \emph{labeled vertices}. 
Following standard notation, we will assume that in the egf of a class of networks the parameter $x$ always marks the number of labeled vertices.

The above description provides us with an explicit relation between the class $\CB$ and the (corresponding) class of networks $\CN$, which we shall
describe now. Let~$\vec{\CB}$ be the class containing ordered pairs $(B\setminus e, e)$ where~$B \in \CB$ and $e \in E(B)$, 
that is, an edge of $B$ is distinguished and removed. 
Note that each $\vec{B}\in\vec{\CB}_n$, except from that where the underlying graph is a single distinguished edge, gives rise to 
two networks with $n-2$ labeled vertices: one is obtained by removing the labels from the endpoints of the distinguished non-edge (and relabeling the remaining vertices such
that only labels in $\{1, \dots, n-2\}$ appear), and the other is obtained by adding the distinguished non-edge to the underlying graph of $\vec{B}$,
removing the labels and performing a relabeling as before. Moreover, the single distinguished edge gives rise to the network that consists 
of two isolated poles. 
With~$\mathcal{X}$ denoting the class consisting of a single labeled isolated vertex, and $e$ being the network consisting of a single edge, 
this implies that the classes~$\CB$ and~$\CN$ are due to the definition of~$\CN$ related through
\begin{equation}\label{eq:decompositionNetworksEdgeRooted2Connected}
(\CN + 1)\times {\mathcal X}^2  = (1+e) \times \vec{\CB}.
\end{equation}
Note that this immediately translates to the relation
$
	\frac{\partial B(x,y)}{\partial y} = \frac{x^2}{2} \frac{1 + N(x,y)}{1 + y},
$
which is the well-known dependency linking the egf's for $\CB$ and (the associated) $\CN$, see also e.g.\ \cite{Walsh}.

\subsection*{Our Setting}
The main setting in this work is the following. We will begin with a class of networks $\CN$, which fulfills certain closure properties that are 
described below. This, combined with~\eqref{eq:decompositionNetworksEdgeRooted2Connected}, defines a class of biconnected graphs $\CB$.
Those $\CB$ are precisely the ones that we will consider here. The following statement, along with our probability estimates, implies that the
structural properties of random networks that occur asymptotically almost surely translate into almost sure properties of random biconnected
graphs. Consequently, in the remainder of the paper we  only need to consider random networks. Let $\Pa$ denote a graph property, that is, a
set of graphs which is closed under automorphisms. 
With a slight abuse of terminology we say that a network $N$ has a property $\Pa$, if the graph  that is 
obtained by putting labels to the poles has property $\Pa$. The proof can be found in Section~\ref{ssec:proofNets}.
\begin{proposition}
\label{prop:transferNB}
Let $\CN$ be a class of networks and let $\CB$ be the class of the corresponding biconnected graphs. 
Moreover, let $\Bl_n$ be a uniform random graph from $\CB_n$, and $\Nl_n$ a network that is drawn uniformly at random from $\N_n$. 
Suppose that
$\Pr{\Nl_{n-2} \in \Pa} \ge 1 - f(n-2)$, where $\Pa$ is any property of graphs that is closed under automorphisms. 
Then $\Pr{\Bl_n \in \Pa} \ge 1 - 2\kappa f(n-2)$, where $\kappa n$ is the maximum number of edges in a graph in $\CB_n$.
\end{proposition}
With the above facts at hand we are ready to define the classes of networks that we will consider. Let $\mathcal{T}$ denote a class of labeled 3-connected graphs that is closed under automorphisms. Following~\cite{Traht, Tutte}, we define a class of networks \emph{with respect to} $\mathcal{T}$, denoted by $\mathcal{N}(\mathcal{T})$, inductively as follows (see also Figures~\ref{fig:network_decomposition} and~\ref{fig:HNEt}): 

\begin{figure}[t]%
\includegraphics[width=13cm]{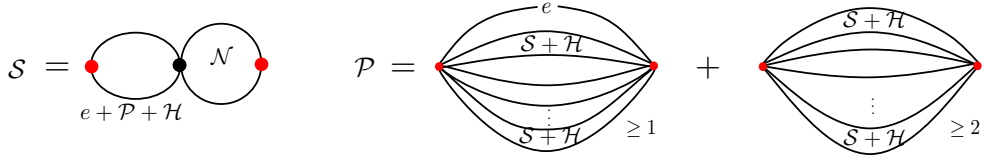}%
\caption{Series and Parallel Networks.}%
\label{fig:network_decomposition}%
\end{figure}

\smallskip

\noindent
A \emph{network} in $\mathcal{N}(\mathcal{T})$ is either an edge, whose endvertices are the poles, or is of type $S$ (\emph{series network}), or of type $P$ (\emph{parallel network}), or of type $H$ (\emph{core network}). 

\smallskip
\noindent
\emph{Series networks}: 
A network of type $S$ consists of two networks $N_1$ and $N_2$, such that the right pole of $N_1$ is identified with the left pole of $N_2$. Here, $N_1$ is restricted to be either an edge, or of type $S$ or $H$, and $N_2\in \mathcal{N}(\mathcal{T})$. 

\smallskip
\noindent
\emph{Parallel networks}:
A network of type $P$ consists either of an edge and a non-empty set of networks, 
either of type $S$ or of type $H$, where 
their right poles (left poles) are identified into a single right pole (left pole), or a set of networks of size at least 2, either of type $S$ or of type $H$ where the identification of the poles is as before.

\smallskip
\noindent
\emph{Core networks:} Let $\bar{\mathcal{T}}$ be the class of networks which are created by taking any graph in $\T$, deleting an edge, and then turning its endvertices into poles. A network of type $H$ consists of a network from $\bar{\mathcal{T}}$, where each edge is replaced by a network whose poles are identified in a unique way with the endvertices of the edges. 

\smallskip
\noindent
In any of the above cases, we do the necessary relabeling in case there are conflicts, when two or more networks are joined through one of the above operations. 
\begin{figure}[b!]%
\includegraphics[width=10cm]{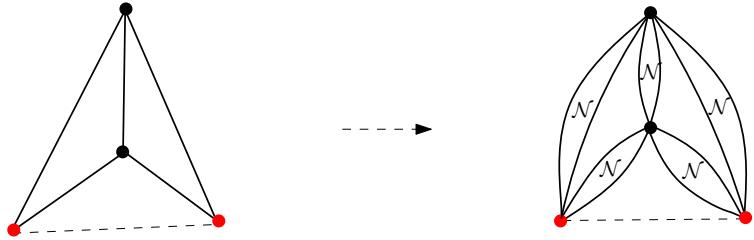}%
\caption{A core network whose core is $K_4$.}%
\label{fig:HNEt}%
\end{figure}
With the above notation at hand, we can now introduce formally the classes of networks that we are interested in.
\begin{definition} \label{Nice}
We say that $\N (\T )$ is $\al$-\emph{nice}, if the following conditions hold.
\begin{enumerate}
\item [(A)] For all $y>0$, $\bar{T}(x,y)$ is analytic in $\mathbb{C}$ or has a unique dominant singularity at $\rho_{\bar T}(y) > 0$, and admits a local singular expansion of the form 
\begin{equation} \label{Ass:SingExp}
\textstyle
\bar{T}(x,y) = \sum\limits_{k \geq 0} t_k(y) \left(1- {x\over \rho_{\bar T}(y)}\right)^{k/m},
\end{equation}
where $m\in\mathbb{N}$, and the $t_k(y)$ are analytic functions. 
Furthermore, there exist $\delta, \eps >0$ such that $\bar{T}(x,y)$ is analytic in a domain
$\Delta = \Delta (\delta, \eps) = \{z \ : \ |z|< \rho_{\bar{T}} (y) + \delta, \arg (z-\rho_{\bar{T}}(y))> \eps \}$.
Moreover, the smallest integer~$k$ such that $t_k(y)\not\equiv 0$ and $k/m \not\in \mathbb{N}$ satisfies $k/m = \alpha$. Also, the function $\rho_{\bar T}(y)$ is strictly decreasing and continuously twice differentiable.
\item[(B)] For $y$ in a neighborhood of 1, let $\rho_N(y)$ be the radius of convergence of the egf enumerating $\N (\T )$ with respect to vertices and edges. Then
\[
\textstyle
	- \frac{\rho_N''(1)}{\rho_N(1)}
	- \frac{\rho_N'(1)}{\rho_N(1)}
	+ \left(\frac{\rho_N'(1)}{\rho_N(1)}\right)^2 \neq 0.
\]
 \end{enumerate}
\end{definition}
The assumptions in the definition above are satisfied by several classes of graphs that have been studied in the literature, e.g.\ by series-parallel, planar or~$K_{3,3}$-minor free graphs. However, we believe that the assumption (B) is redundant, in the sense that it follows from (A), but we are unfortunately unable to show it. The remaining conditions are minimal in the context of analytic combinatorics, as they form the ``backbone'' of essentially any counting result or limit theorem that can be derived.

For notational convenience and in slight abuse of notation we will be writing~$T(x,y)$ for the egf~$\bar{T}(x,y)$. With all the above definitions at hand we are now ready to present our main results.
Given a network~$N$, we denote by~$C_1(N)$ the number of vertices in a largest core of~$N$. If~$\ell$ is a positive integer and~$N$ is a network, we denote by~$c(\ell;\, N)$ the number of cores in~$N$ which have~$\ell$ vertices. More generally, if~$\xi \geq 1$, we denote by~$c(\ell, \xi \ell;\, N)$ the number of cores in~$N$ whose number of vertices is at least ~$\ell$ and no more than~$\xi \ell$. Set
\begin{equation}
\label{eq:Phi}
	\Phi( x, y, z) = T(x,z) - \log\left(\frac{1+z}{1+y}\right) + \frac{xz^2}{1+xz}.
\end{equation}
An important property of this function is that it defines implicitly the egf enumerating networks through the relation~$\Phi(x,y,N(x,y)) = 0$; see Lemma~\ref{lem:egfsNetworks}. Our results imply that the sign of~$\frac{\partial}{\partial z}\Phi(\rho_N, 1, N_0) =: \lambda$ dictates whether a random network has only at most logarithmically-sized cores ($\lambda > 0$), or a giant core that contains a constant fraction of the vertices ($\lambda < 0$). If~$u \in \{x,y,z \}$, we shall denote in the remainder by~$\Phi_u(x,y,z)$ the partial derivative of~$\Phi (x,y,z)$ with respect to~$u$. Moreover, we denote by~$\Nl_n$ a random network from~$\N_n$ throughout and without further reference. In the case~$\lambda > 0$ we show the following. 
\begin{theorem} \label{Main:Subcritical}
Let $\N(\T)$ be an $\alpha$-nice class for some $\alpha\in \mathbb{R}\setminus \{0\}$, and let~$\delta, \eps > 0$. Assume that $\Phi_z(\rho_N, 1 , N_0) > 0$ and set $\tau = {\rho_N\over \rho_T (N_0)}$. Then the following holds a.a.s.
\begin{enumerate}
\item [(i)]  For all $2\leq \ell \leq (1-\delta)\log_{1/\tau} n$ we have $c(\ell;\, \Nl_n) \in (1\pm \eps) c_{\ell} n$, where $c_{\ell}$ is given in Lemma~\ref{lem:counts}.
\item [(ii)] $c\left( (1-\eps)\log_{1/\tau} n,  {5\over 2} \log_{1/\tau} n;\, \Nl_n\right)\leq n^{2\eps}$.
\item[(iii)] The largest core in $\Nl_n$ contains at most $(5/2 +\delta) \log_{1/\tau} n$ vertices.
\end{enumerate}
Moreover, for large $\ell$ the Equation~\eqref{eq:pk_subcritical_asympt} provides an asymptotic expression for $c_\ell$.
\end{theorem}
Our next theorem deals with the case $\lambda < 0$, where it is shown that the 
corresponding random networks have a.a.s.\ a unique giant core.
\begin{theorem} \label{Main:Critical} 
Let $\N(\T)$ be an $\alpha$-nice class for some $\alpha\in \mathbb{R}\setminus \{0\}$, and let~$\delta, \eps > 0$. Assume that $\Phi_z(\rho_N, 1 , N_0) < 0$. 
Let $\omega(n)$ be a function such that $\omega (n) \rightarrow \infty$ as $n \rightarrow \infty$.
Then the following holds a.a.s.
\begin{enumerate} 
\item [(i)] Let $\gamma_{T}>0$ be as in Lemma~\ref{lem:uniquelarge}. Then $C_1 (\Nl_n) \in (1\pm \eps)\gamma_T n$. 
\item [(ii)] For all $\left(n \omega(n) \right)^{1/\alpha}\leq \ell < C_1(\Nl_n)$ we have $c(\ell;\, \Nl_n) = 0$.
\item [(iii)] For all $2\leq \ell  \leq \left( {n \over \omega(n) \log n}\right)^{1 /( \alpha +1)}$ we have $c(\ell;\, \Nl_n) \in (1\pm \eps) c_{\ell} n$, where~$c_{\ell}$ is given in Lemma~\ref{lem:counts}.
\item [(iv)] Let~$\xi \geq 1$. If~$2\leq \ell \leq  \left( {n \over \omega(n) \log n}\right)^{1/\alpha}$, then~$c(\ell, \xi \ell;\, \Nl_n) \in (1\pm \eps) c_{\ell,\xi}  n$, where~$c_{\ell,\xi}$ is given in Lemma~\ref{eq:p_kxik}.
\end{enumerate}
Moreover, for large $\ell$ the Equations~\eqref{eq:pk_asympototic},~\eqref{eq:pkxik} provide asymptotic expressions for $c_\ell$ and~$c_{\ell, \xi\ell}$.
\end{theorem}
Recently, Gim\'enez, Noy and Ru\'e~\cite{GimNoy} obtained 
independently results regarding the asymptotic distribution of the largest core in a random planar graph. 
Their methods are completely different and are based on techniques that are widely used in analytic combinatorics.
However, in the present work, we obtain a fairly precise picture of the typical structure of a random biconnected graph whose 3-connected 
cores are taken from various general families, and we perform a very precise census of the smaller cores in all possible analytic regimes.

Apart from the structural results concerning the sizes of the cores  of~$\Nl_n$, we also obtain a limit law for the number of edges of~$\Nl_n$. 
\begin{theorem} \label{thm:edges}
Let~$\N(\T)$ be an~$\alpha$-nice class of networks with respect to a class~$\T$. Let~$\Nl_n$ be a random network with~$n$ vertices. Then~$e(\Nl_n)$ satisfies a central limit theorem, where the expected value is $-\frac{\rho_N'(1)}{\rho_N(1)}n + o(n)$, and the variance is proportional to~$n$. Moreover, there exists~$\eps_0>0$ and a~$C>0$ such that for all~$0\le \eps < \eps_0$ we have 
$$\prob \left(\left|e(\Nl_n) - \ex (e(\Nl_n))  \right|> \eps \ex (e(\Nl_n))\right) \leq \exp \left(- C\eps^2 n \right).$$
\end{theorem}

\section{Preliminaries}
\input{preliminaries}


\section{Singularity Analysis for Networks}
\label{sec:singAnalysis}

A typical situation that is commonly encountered in modern scenarios of asymptotic enumeration is the following (see also~\cite{FlajSed}): we want to determine the precise asymptotic behavior of the coefficients of an unknown generating function $F(x)$, which is given implicitly by some functional equation $G(x, F(x)) = 0$. Of course, only in very seldom cases it is possible to obtain $F(x)$ explicitly, and generally, we have to resort to a completely different ``program'' to extract enough information from the given functional equation, so as to be able to understand the behavior of~$[x^n]F(x)$.

The fundamental principle behind this program is that the knowledge of the behavior of an analytic function in the vicinity of its dominant singularities gives us very precise information about the asymptotics of its coefficients. So, we have to accomplish two tasks: first, to determine the dominant singularities of $F(x)$, and second, to find the asymptotic expansion of our function near those singular points. The main objective of this section is to accomplish this program for the classes of $\alpha$-nice networks. 




~\\
According to Lemma~\ref{lem:egfsNetworks} the function~$N(x,y)$ is given by the solution of~$\Phi(x,y,N(x,y)) = 0$, where~$\Phi$ is as in~\eqref{eq:Phi}. Let~$y>0$ be any fixed value. As~$N(x,y)$ has only non-negative coefficients, by applying Pringsheim's Theorem (see e.g.\ \cite{FlajSed}), we infer that if the radius of convergence of~$N(x,y)$ is~$\rho_N(y)$, then~$\rho_N(y)$ is a singularity of~$N(x,y)$. 

Note that~$\Phi(x,y,z)$ is not analytic at points that satisfy~$x > \rho_T(z)$, as~$T(x,z)$ ceases to be analytic there. Hence, we always have that~$\rho_N(y) \le \rho_T( N_0(y))$. Moreover, suppose that for some pair $(x',z')$ such that $x' < \rho_T(z')$ we have $\Phi(x',y,z') = 0$ and simultaneously $\Phi_z(x',y,z') \neq 0$. Then, by the Implicit Function Theorem, see e.g.\ \cite{FlajSed}, $\Phi$ is locally invertible, which means that there exists an analytic continuation of $N(x,y)$ around $x = x'$, and $N(x',y) = z'$. We can therefore obtain information about~$\rho$ by studying the distribution of the zeros of~$\Phi_z(x,y,z)$, i.e., the points around which~$\Phi$ fails to be invertible.

It turns out that the precise localization and nature of the singularity of~$N(x,y)$ is dictated by one of the following analytic conditions:
\begin{itemize}
	\item \emph{Subcritical case}: We have~$\Phi_z(\rho_N(y), y, N_0(y)) = 0$. In this case,~$\Phi$ has a \emph{branch point} 
at~$(\rho_N(y), y, N_0(y))$. Moreover, if~$\Phi$ was analytic at~$(\rho_N(y), y, N_0(y))$, then it readily follows 
that~$\rho_N(y) < \rho_T( N_0(y))$.	We further show in Lemma~\ref{lem:branchPoint} that under certain mild assumptions the singularity of
$N(x,y)$ is of square-root type.
	\item \emph{Supercritical case}: We have~$\Phi_z(\rho_N(y), y, N_0(y)) \neq 0$. Here, the function~$\Phi$ is not analytic at~$(\rho_N(y), y, N_0(y))$, and inevitably we have that~$\rho_N(y) = \rho_T( N_0(y))$. In this setting, we determine the conditions 
(see Theorem~\ref{thm:singTransferGeneral} for a precise statement) under which the singularity type of some implicitly given function is ``inherited'' from the singularity of the implicit function. 
\end{itemize}
The following lemma will become very handy when we study the set of zeros of $\Phi_z$.
\begin{proposition}
\label{prop:PhiPositive}
Let $y>0$ and suppose that a dominant singularity of $N(x,y)$ is given by~$\rho_N(y)$. Let $\rho_\theta = \rho_N(y)e^{i\theta}$. Then, for any $0 < \theta < 2\pi$
\[
	|\Phi_z(\rho_\theta, y, N(\rho_\theta, y))| > -\Phi_z(\rho_0, y, N(\rho_0,y)).
\]
\end{proposition}
\begin{proof}
A simple calculation shows that $\Phi_z(x,y,z) = T_y(x,z) - \frac{1 - xz^2(2+xz)}{(1+z)(1+xz)^2}$. Let us abbreviate~$N(\rho_\theta, y) = N_\theta$. By applying the inequality $|a - b| \ge |b| - |a|$, which is valid for all complex numbers $a,b$, we obtain
\[
	|\Phi_z(\rho_\theta, y, N(\rho_\theta, y))|
	\ge
	\left|\frac{1 - \rho_\theta N_\theta^2(2 + \rho_\theta N_\theta)}{(1 + N_\theta)(1 + \rho_\theta N_\theta)^2}\right|
	- |T_y(\rho_\theta, N(\rho_\theta, y))|.
\]
As $T(x,y)$ and $N(x,y)$ have only non-negative coefficients, and $T$ is aperiodic, the triangle inequality implies that
\[
	|T_y(\rho_\theta, N_\theta)|
	< T_y(|\rho_\theta|, |N(\rho_\theta, y)|)
	\le T_y(|\rho_\theta|, N(|\rho_\theta|, y))
	= T_y(\rho_0, N_0).
\]
Similarly, we obtain that
\[
	|1 - \rho_\theta N_\theta^2(2 + \rho_\theta N_\theta)|
	\ge
	1 - |\rho_\theta N_\theta^2(2 + \rho_\theta N_\theta)|
	\ge
	1 - \rho_0 N_0^2(2 + \rho_0 N_0),
\]
and,
\[
	|(1 + N_\theta)(1 + \rho_\theta N_\theta)^2| \le (1 + N_0)(1 + \rho_0 N_0)^2.
\]
The proof completes by putting everything together.
\end{proof}
The remainder of this section deals separately with the subcritical and the critical case.
\subsection*{The Subcritical Case}
Our result in this section says that if $\Phi$ ceases to be invertible at some analytic point inside its domain, then the egf enumerating networks has a singularity of the square-root type.
\begin{lemma} 
\label{lem:branchPoint}
Let $\Phi(x,y,z)$ be as in~\eqref{eq:Phi}, and let $Y$ be a compact subset of $(0, +\infty)$. Suppose that for any $y_0\in Y$ there exists a minimal $N_0>0$ and a minimal $0 < x_0 < \rho_T(N_0)$ such that
\begin{equation}
\label{eq:PhiBranch}
	\Phi(x_0, y_0, N_0) = 0
	\enspace \text{ and } \enspace
	\Phi_z(x_0, y_0, N_0) = 0.
\end{equation}
Then there exists an $\eps > 0$ such that $N(x,y)$ admits a local representation
\[
	N(x,y) = g(x,y) - h(x,y)\sqrt{1 - \frac{x}{\rho_N(y)}},
\]
for $y\in Y$, $|x - \rho_N(y)| < \eps$ and $|\text{arg}(x - \rho_N(y))| \neq 0$, where $g(x,y), h(x,y)$ and $\rho_N(y)$ are analytic around $x = x_0$ and $y = y_0$, and $\rho_N(y_0) = x_0$ and $g(x_0,y_0) = N_0$. Moreover, for any~$y\in Y$,~$N(x,y)$ is analytic in a domain $\Delta(\rho(y))$.
\end{lemma}
\begin{remark}
Proofs of similar statements were found in the literature for e.g.\ the egf enumerating series-parallel networks~\cite{ar:bgkn05}. However, the proof of the above claim requires more work: the function $\Phi$ turns out to have in general negative coefficients, which makes the application of standard tools (see e.g.~\cite{Dmorta}) for this purpose impossible. Nevertheless, the proof that is presented here makes more or less use of standard techniques, and we include it for completeness.
\end{remark}
\begin{proof}
Let us first collect some basic properties of $\Phi$. As $T$ has non-negative coefficients, for any $x,y \ge 0 $ and $z> 0$
\[
	\Phi_{zz}(x,y,z) = T_{zz}(x,z) + \frac{1+7xz+3x^2z^2+x^3z^3+2x+2xz^2}{(1 + z)^2 (1 + xz)^3}>0,
\]
and
\begin{equation}
\label{eq:Phi_x}
	\Phi_x(x,y,z) = T_{x}(x,z) + \frac{z^2}{(1+xz)^2} >0.
\end{equation}
As $\Phi(x_0,y_0,N_0) = \Phi_z(x_0, y_0, N_0) = 0$ and $\Phi_{zz}(x_0, y_0, N_0) \neq 0$, the Weierstrass Preparation Theorem guarantees the existence of a function~$H(x,y,z)$, which is analytic around~$(x_0,y_0,N_0)$ and~$H(x_0,y_0,N_0)\neq 0$, and two functions~$p(x,y), q(x,y)$ that are analytic around~$(x_0, y_0)$ and satisfy $p(x_0, y_0) = q(x_0, y_0) = 0$, such that
\begin{equation}
\label{eq:PhiafterWeiserstrass}
	\Phi(x,y,z) = H(x,y,z)\left((z-N_0)^2 + p(x,y)(z - N_0) + q(x,y)\right).
\end{equation}
So, every function that satisfies $\Phi(x,y,N(x,y)) = 0$ is given in a neighborhood of $(x_0,y_0)$ by one of the two functions
\[
	N(x,y) = N_0 - \frac{p(x,y)}{2} \pm \sqrt{\frac{p(x,y)^2}{4} - q(x,y)}.
\]
Note that $\frac{p(x_0,y_0)^2}{4} - q(x_0,y_0) = 0$. Moreover, by differentiating both sides of~\eqref{eq:PhiafterWeiserstrass} by $x$ we obtain that $q_x(x_0,y_0) \neq 0$, as $\Phi_x(x_0, y_0, N_0) \neq 0$. Hence $\frac{p_x(x_0,y_0)^2}{4} - q_0(x_0,y_0) \neq 0$, and again the Preparation Theorem guarantees the existence of a function $K(x,y)$, which is analytic around $(x_0, y_0)$ and $K(x_0,y_0) \neq 0$, and a function $r(y)$, which is analytic around $y_0$ and $r(y_0) = 0$, such that
\[
	\frac{p(x,y)^2}{4} - q(x,y) = K(x,y)(x - x_0 + r(y)).
\]
By putting everything together we obtain that $N$ has a local representation around $(x_0, y_0)$ of the kind
\[
	N(x,y) = g(x,y) \pm h(x,y)\sqrt{1 - \frac{x}{\rho_N(y)}},
\]
where $g,h$ are analytic around $(x_0, y_0)$ and $g(x_0,y_0) = N_0$, and $\rho_N(y)$ is analytic around $y_0$ and $\rho_N(y_0) = x_0$. Moreover, we choose $h$ without loss of generality such that $h(x_0, y_0) > 0$. 

We will argue next that the ``$-$''-sign above is the right choice. Denote the solution with the~$-$-sign by $N_1$, and the other one by $N_2$. Note that by the minimality of $x_0$ and $N_0$ we have $\Phi_z(x, y,N) \neq 0$ for $0 \le x < x_0$ and $0 \le N< N_0$ such that $\Phi(x,y,N) = 0$, which implies that there is a unique analytic function that satisfies $\Phi(x, y, N(x,y)) = 0$ around $0$. Moreover, note that the function $N_1$ is increasing as $x$ approaches $x_0$ from left, while $N_2$ is decreasing. To show that the increasing branch is the function we are looking for, we show that there is a analytic curve that connects $N_1$ to the function that is analytic around $x = 0$.

We show the claim by contradiction. Let $y_0\in Y$. First, suppose that there is a $0 < \tilde x_0 < x_0$ such that $N(\tilde x_0, y_0) = N_0$. Then $\Phi(\tilde x_0, y, N_0) = \Phi(x_0, y_0, N_0) = 0$, which is a contradiction, as $\Phi$ is monotone with respect to its first variable, due to~\eqref{eq:Phi_x}. On the other hand, suppose that there is a $\tilde N_0 < N_0$ such that $N(x_0, y_0) = \tilde N_0$. Note that $\Phi_z(x_0,y_0,0) = -1 < 0$, and due to the minimality of $N_0$, for all $0< N < N_0$ we have that $\Phi_z(x_0,y_0,N) < 0$. This again contradicts the assumption $\Phi(x_0,y_0,\tilde N_0) = \Phi(x_0,y_0,N_0)$. In other words, we have shown that the the solution of $\Phi(x,y_0,N(x,y_0)) = 0$ never ``leaves'' the rectangle $\{(x,N) ~|~ 0 \le x \le x_0, 0\le N\le N_0\}$, thus implying that that the increasing branch analytically continues the solution around 0.

To complete the proof we will show that $N(x,y)$ is analytic in a $\Delta$-domain. By applying Proposition~\ref{prop:PhiPositive} we obtain for \emph{any} $\rho_\theta = \rho_N(y)e^{i\theta}$, where $0 < \theta < 2\pi$, that
\[
	\Phi_z(\rho_\theta, y, N(\rho_\theta, y))
	> -\Phi_z(\rho_0, y, N(\rho_0, y))
	= 0.
\]
Thus, the Implicit Function Theorem shows that there is an analytic continuation of $N(x,y)$ around $\rho_\theta$, where $0 < \theta<2\pi$. This completes the proof.
\end{proof}

\subsection*{The Critical Case \& Transfer of Singular Expansions}
The following result provides us with a local expansion of a function $f$ that is given implicitly by a functional equation $f = G(x,y,f)$. In contrast to the previous result, here we handle general functions $G$ at points where they \emph{are not} analytic, thus extending a result from~\cite{Dmorta}, where just the case of singular behavior of type $3/2$ was treated. This result will be used in the proof of Lemma~\ref{lem:PhiSingular}, which complements the statement of Lemma~\ref{lem:branchPoint} in the case where the function $\Phi$ is invertible in its whole disc of convergence.
\begin{theorem}
\label{thm:singTransferGeneral}
Suppose that $G(x,y,f)$ has a local representation of the form
\[
	G(x,y,f) = g(x,y,f) + h(x,y,f)\left(1 - \frac{f}{r(x,y)}\right)^{{k}/{m}},
\]
where $k,m\in\mathbb{N}$, $1 < k/m \not\in\mathbb{N}$, the functions $g,h,r$ are analytic around $(x_0, y_0, f_0)$, and satisfy
\[
	g_f(x_0,y_0,f_0) \neq 1,
	\enspace
	h(x_0,y_0,f_0) \neq 0,
	\enspace
	r(x_0, y_0) \neq 0,
	\enspace
	r_x(x_0, y_0) \neq g_x(x_0,y_0,f_0).
\]
Suppose that $f = f(x,y)$ is a solution of the functional equation $f = G(x,y,f)$ such that $f(x_0, y_0) = f_0$. Then $f$ admits a local representation of the form
\begin{equation}
\label{eq:fsingExp}
	f(x,y) = \tilde{g}(x,y) + \tilde{h}(x,y)\left(1 - \frac{x}{\rho(y)}\right)^{{k}/{m}},
\end{equation}
where $\tilde{g}, \tilde{h}, \rho$ are analytic at $(x_0,y_0)$ and $\tilde{h}(x_0,y_0) \neq 0$ and $\rho(y_0) = x_0$.
\end{theorem}
The next lemma takes care of the cases in which $\Phi$ is not analytic at $(\rho_N(y),y,N(\rho_N(y),y))$.
\begin{lemma}
\label{lem:PhiSingular}
Let $\Phi(x,y,z)$ be as in~\eqref{eq:Phi}, and let $Y$ be a compact subset of $(0, +\infty)$. Suppose that for all $y_0\in Y$ and all $N_0 > 0$ and $0 < x_0 \le \rho_T(N_0) $ such that $\Phi(x_0, y_0, N_0) = 0$ it holds
\[
	\Phi_z(x_0, y_0, N_0) \neq 0.
\]
Then, if $T(x,y)$ admits a local representation of the form~\eqref{Ass:SingExp} and is analytic in a $\Delta$-domain, there exists an $\eps > 0$ such that $N(x,y)$ admits a local representation
\[
	N(x,y) = g(x,y) + h(x,y)\left(1 - \frac{x}{\rho_N(y)}\right)^{k/m},
\]
for $y\in Y$, $|x - \rho_N(y)| < \eps$ and $|\text{arg}(x - \rho_N(y))| \neq 0$, where $g(x,y), h(x,y)$ and $\rho_N(y)$ are analytic around $x = x_0$ and $y = y_0$, and $g(\rho_N(y_0),y_0) = N_0$. Moreover, for any $y\in Y$,~$N(x,y)$ is analytic in a domain $\Delta(\rho_N(y))$.
\end{lemma}
\begin{proof}
First of all, note that the assumption on $T(x,y)$ implies that there are two functions $g(x,y,z)$ and $h(x,y,z)$, which are analytic at $(x_0, y_0, N_0)$, such that locally 
\[
	\Phi(x,y,z) = g(x,y,z) + h(x,y,z)\left(1 - \frac{x}{\rho_T(z)}\right)^{k/m}.
\]
Our aim is to apply Theorem~\ref{thm:singTransferGeneral} in order to determine a local expansion for $N(x,y)$, which is implicitly defined by $\Phi(x,y,N(x,y)) = 0$. To achieve this, we shall first show that we can ``switch'' in the above expression between local expansions in terms of $x$ and $z$. Moreover, we further will check the remaining conditions in Theorem~\ref{thm:singTransferGeneral}.

Recall that $\rho_T$ is a strictly decreasing function that attains only positive values. Hence we have that $\rho_T(N_0) \neq 0$ and $\rho_T'(N_0) \neq 0$. The Weierstrass preparation theorem implies that there is a function $H(x,z)$, which is analytic at $(x_0, N_0)$ and $H(x_0, N_0) \neq 0$, such that
\[
	\rho_T(z) - x = H(x,z)(r(x) - z),
\]
where $r(x)$ is the analytic inverse function of $\rho_T$ in a neighborhood of $x_0$. Hence we infer that there is an analytic function $\tilde{h}(x,y,z)$ such that
\[
	\Phi(x,y,z) = g(x,y,z) + \tilde{h}(x,y,z)\left(1 - \frac{z}{r(x)}\right)^{k/m}.
\]
It is now routine to check from our assumptions the remaining preconditions of Theorem~\ref{thm:singTransferGeneral}. Hence, we obtain the claimed local representation for $N$.

To complete the proof we show that $N$ is analytic in a $\Delta$-domain. For the sake of contradiction, suppose that there is another singularity on the circle of convergence of $N(x,y)$, and write $\tilde\rho_\theta = \rho_N(y) e^{i\theta}$, where $0 < \theta < 2\pi$. Note that $\Phi$ is analytic at $(\tilde\rho(y),y,N(\tilde\rho(y),y))$, because $T$ is analytic there. So, we must have $\Phi_z(\tilde\rho(y),y,N(\tilde\rho(y),y)) = 0$, as otherwise the Implicit Function Theorem would guarantee the existence of an analytic continuation of $N$ around $\rho_\theta$. Then, by applying Proposition~\ref{prop:PhiPositive} we obtain
\begin{equation}
\label{eq:phiCritical}
	0 = \Phi_z(\rho_\theta,y,N(\rho_\theta,y)) > -\Phi_z(\rho_0,y,N(\rho_0,y))
\end{equation}
Note that $\Phi(0,y,y) = 0$, and hence $N(0,y) = y$ (this follows also from the definition of networks, as there is precisely one network that consists of a single edge). A straightforward calculation shows that $\Phi_z(0,y,y) = -\frac{1}{1+y} < 0$. Hence, due to our assumptions we have for all $0\le x\le x_0$ that $\Phi_z(x,y,N(x,y)) < 0$, which implies that~\eqref{eq:phiCritical} is a contradiction.
\end{proof}

\begin{proof}[Proof of Theorem~\ref{thm:edges}]
Let $y$ be in some fixed complex neighborhood of 1. By combining Lemma~\ref{lem:branchPoint} with Lemma~\ref{lem:PhiSingular} we obtain that there are analytic functions $g, h$ such that locally around $\rho_N(y)$
\begin{equation}
\label{eq:singExpNGeneral}
	N(x,y) = g(x,y) + h(x,y)\left(1 - \frac{x}{\rho_N(y)}\right)^{r},
\end{equation}
where $r = 1/2$ if $\Phi(\rho_N(y), y, N_0(y)) > 0$, and $r = \alpha$ if $\Phi(\rho_N(y), y, N_0(y)) < 0$. By applying the Transfer Theorem (see Corollary VI.1 in~\cite{FlajSed}) we obtain that the probability generating function $p_n(u)$ for the number of edges in $\Nl_n$ satisfies
\[
	p_n(u) = h(x,u)\frac{n^{-r-1}}{\Gamma(r)}\rho_N(u)^{-n} (1 + o(1)).
\]
The proof finishes by applying the Quasi-Powers Theorem (see Theorem IX.8 in~\cite{FlajSed}), and by exploiting Propoerty (B) in the definition of nice classes of networks.
\end{proof}


\section{Sampling Networks} \label{Sec:Samplers}

Although the existence of Boltzmann samplers for several classes of networks is guaranteed by the work of Duchon et al.~\cite{Duch}, we shall
 nevertheless describe them explicitly in this section. This will not only allow us to introduce some necessary notation that will be used in the rest
 of the paper, but it will also make possible to relate the number of 3-cores in networks to certain random decisions performed by the samplers. We
 will make such statements more precise in Section~\ref{ssec:combinatoricsSamplers}.

\subsection{Boltzmann Samplers for Networks}
Using the rules that are described in Section~\ref{ssec:boltzmannSampling} we construct the Boltzmann samplers for several classes of networks.

We begin with the class~$\CN$. Note that any network in~$\CN$ is either an edge, or an~$\CS$-network, or a~$\CP$-network, or a~$\CH$ network. Hence, a Boltzmann sampler for~$\CN$ will call a sampler for a subclass with probability proportional to the value of the generating function of this subclass. More precisely, we say that a variable~$X$ is \emph{network-distributed} with parameters~$x$ and~$y$,~$X\sim\Net(x,y)$, if its domain is the set of symbols~$\Omega_{\Net} = \{e, S, P, H\}$ and for any~$s\in\Omega_\Net$ it holds~$\prob(X = s) = \frac{s(x,y)}{N(x,y)}$. Then the sampler for~$\CN$ can be described concisely as follows.

	\medskip
	\begin{tabular}{lll}
		$\Gamma N(x,y)$ :	&~$s \leftarrow \Net(x,y)$ \\
		& \RETURN~$\Gamma s(x,y)$ \\
	\end{tabular}
	\medskip

Next we describe the sampler for ~ $\CS$. For two networks~$N_1$ and~$N_2$ we will write~$N_1N_2$ for the network that is obtained by identifying the right pole of~$N_1$ with the left pole of~$N_2$. Recall that~$\CS = (e + \CP + \CH) \times \CN$. Hence, $\Gamma S(x,y)$ has to choose among the classes~$e$,~$\CP$, and~$\CH$ with the right probability. More precisely, we say that a variable~$X$ is \emph{series-distributed} with parameters~$x$ and~$y$,~$X\sim\Ser(x,y)$, if its domain is the set of symbols~$\Omega_{\Ser} = \{e, P, H\}$ and for any~$s\in\Omega_\Ser$ it holds~$\prob(X = s) = \frac{s(x,y)}{S(x,y)}$. Then the sampler for $\CS$ is given by the following procedure.

	\medskip
	\begin{tabular}{lll}
		$\Gamma S(x,y)$ :	& $s \leftarrow \Ser(x,y)$ & \\
		& $\ell \leftarrow s(x,y)$ & \\
		& $r \leftarrow N(x,y)$ & \\
		& \RETURN $\ell r$, relabeling randomly its vertices \\
	\end{tabular}
	\medskip

In the sequel we describe the sampler for $\CP$. A parallel network either consists of an edge and a set of at least one $\CS$ or $\CH$ network ($= \Set_{\ge 1}(\CS + \CH)$), or it consists of a set of at least two $\CS$ or $\CH$ networks ($= \Set_{\ge 2}(\CS + \CH)$). This implies that $\Gamma P(x,y)$ has to first to choose one of the two possibilities with the right probability, and then to sample a set (with a given lower bound on the number of elements) from $\CS + \CH$ according to the Boltzmann distribution.

Let us introduce some notation before we describe formally the sampler. We say that a  variable $X$ is \emph{parallel-distributed} with parameters $x$
and $y$, and write $X\sim\Par(x,y)$, if $X \sim 1 + \Be(\frac{e^{S(x,y) + H(x,y)}-1-S(x,y)-H(x,y)}{P(x,y)})$. Finally, we say that a variable is 
\emph{sh-distributed} with parameters $x$ and $y$, $X\sim\sh(x,y)$, if its domain is the set of symbols $\Omega_{\sh} = \{S, H\}$ and for
$s\in\Omega_{\sh}$ it holds $\BP(X = s) = \frac{s(x,y)}{S(x,y) + H(x,y)}$. Then $\Gamma P$ works as follows.

	\medskip
	\begin{tabular}{lll}
		$\Gamma P(x,y)$ :	& $p \leftarrow \Par(x,y)$ & \\
		& $k \leftarrow \mathtt{Po}_{\ge p}(S(x,y)+H(x,y))$ \vspace{1mm} &\\
		& \FOR $i=1\dots k$ &\\
		& \quad $b_i \leftarrow \sh(x,y) $ & \\
		& \quad $p_i \leftarrow \Gamma b_i(x,y)$ \vspace{1mm} & \\
		& construct a network $P$ by identifying the left and right poles of $p_1, \dots, p_k$ & \\
		& relabel randomly the vertices of $P$ \vspace{1mm} & \\
		& \IF $p = 1$ \THEN \RETURN $P$, where the poles are joined by an edge & \\
		&\ELSE \RETURN $P$ & \\
	\end{tabular}
	\medskip

Finally, we describe the sampler for $\CH$. Recall that a $\CH$-network is obtained by substituting the edges of some graph from $\CT$ by graphs from
$\CN$. Here we will assume that we have an auxiliary sampler $\Gamma T(x,y)$, which samples graphs from $\CT$ according to the Boltzmann 
distribution. Then the sampler for $\CH$ can be described as follows.

	\medskip
	\begin{tabular}{lll}
		$\Gamma H(x,y)$ :	& $T \leftarrow \Gamma T(x, N(x,y))$ \vspace{1mm} &\\ 
		& \FOREACH edge $e$ of $T$ &\\ 
		& $\quad \gamma_e \leftarrow \Gamma N(x,y)$ \vspace{1mm} & \\ 
		& replace every $e$ in $T$ by $\gamma_e$  & \\
		& \RETURN $T$, relabeling randomly its vertices
	\end{tabular}
	\medskip

This completes the definitions of all samplers that we shall exploit. From the theory developed in Duchon et al.~\cite{Duch} 
we immediately obtain the following statement.
\begin{lemma}
\label{lem:boltzmannSamplers}
Let $N\in\CN$ such that $N$ has $n$ labeled vertices and $m$ edges. Then $\BP(\Gamma N(x,y) = N) = \frac{x^ny^{e}}{n!N(x,y)}$. The same statement is also true for networks in the classes $\CS, \CP, \CH$ and the corresponding samplers given above.
\end{lemma}
Using the definition of networks we obtain readily the following property.
\begin{lemma}
\label{cor:GammaNIsEfficient}
Let $\CN(\CT)$ be an $\alpha$-nice class of networks such that $\alpha \in \mathbb{R}\setminus\{0\}$.
Let 
$$\beta = \left\{\begin{array}{ll} 
5/2     & \mbox{, if $\Phi_z (\rho_N,1, N_0) > 0 $}\\
\alpha & \mbox{, if $\Phi_z (\rho_N,1, N_0) < 0 $}
\end{array}  \right. .
$$
Then there is a $b > 0$ such that
$\BP(\Gamma N(\rho_N,1) \in \CN_n) \sim bn^{-\beta}.$
\end{lemma}
\begin{proof}
The proof follows from \ref{eq:singExpNGeneral} for the special case $y = 1$ and the Transfer Theorem (Corollary VI.1 in~\cite{FlajSed}).
\end{proof}
\subsection{Combinatorics of the Samplers}
\label{ssec:combinatoricsSamplers}

Although the Boltzmann samplers are randomized algorithms, we are going to analyze them as if they were deterministic algorithms. In particular, 
we are going to assume that all their parameters are read from lists, which contain independent random samples of the appropriate parameters.


Let $\ell_\Net$ be an infinite list of values from $\Omega_\Net$.
Similarly, $\ell_\Ser$, $\ell_\Par$ and $\ell_\sh$ denote infinite lists of values from $\Omega_\Ser$, $\Omega_\Par$ and $\Omega_\sh$, respectively, and $\ell_j$ list of integers $\ge j$ for $j\in \{1,2\}$. Finally, let $\ell_\CT$ be a sequence of graphs from $\CT$.

Note that if we choose the values in the lists according to the corresponding probability distributions (i.e., the values in $\ell_\Net$ according to independent network-distributed variables with parameters $x$ and $y$, the values in $\ell_\Ser$ according to independent series-distributed variables, \dots), then the probability that the deterministic counterparts of the samplers generate a network is equal to the probability in the Boltzmann model.

Using this observation, the proofs of our main results proceed according to the following schema.
\begin{enumerate}
	\item Relate properties of a network generated by a sampler to properties of the values in the lists $\ell_\Net, \ell_\Ser, \ell_\Par, \ell_\sh, \ell_{1}, \ell_{2}$ and $\ell_\CT$.
	\item Show that the desired properties are observed with high probability, if the values in the lists are chosen independently according to the corresponding probability distributions.
\end{enumerate}
The next two statements take care of step (1). More precisely, the following statement gives us the number of 3-cores in a network generated by $\Gamma N$. Its proof is straightforward and therefore omitted.
\begin{lemma}
\label{lem:countk}
Suppose that the network $N$ was generated by $\Gamma N$ by using the first $a_\CT$ graphs in $\ell_\CT$. Then
\[
	c(k; N) = \left|\left\{1 \le i \le a_\CT ~\big|~ v(\ell_\CT[i]) = k-2\right\}\right|.
\]
\end{lemma}
The next claim gives us some general relations among the values used by the sampler, which will be very useful later. The proof can be found in Section~\ref{ssec:combinatoricsSamplers}, and can be performed by looking closer at how the sampler constructs a large network out of smaller ones. We denote by $\mathbf{1}(E)$ the indicator variable for the event $E$, i.e., $\mathbf{1}(E) = 1$ if $E$ occurs, and $\mathbf{1}(E) = 0$ otherwise.
\begin{lemma}
\label{lem:relations}
Suppose that the network $N$ was generated by $\Gamma N$ by using the first $A_s$ values in $\ell_s$, where $s \in \{\Net, \Ser, \Par, \sh, 1, 2, \CT\}$. Then the following statements are true.
\begin{eqnarray} 
v(N) &=& A_\Ser + \sum_{1\le i\le A_\CT} v(\ell_\CT[i]), \label{eq:nSvT} \\
e(N) &=& \sum_{1\le i\le A_\Net} \mathbf{1}(\ell_\Net[i] = e) + \sum_{1\le i\le A_\Ser} \mathbf{1}(\ell_\Ser[i] = e) + 
            \sum_{1\le i\le A_\Par}\mathbf{1}(\ell_\Par[i] = 1),  \label{eq:edges}   \\
 A_\Net &=& 1+ A_\Ser + \sum_{1\le i \le A_\CT} e(\ell_\CT[i]), \label{eq:aNaSeT} \\
A_j  &=&  \sum_{1\le i\le A_\Par}\mathbf{1}(\ell_\Par[i] = j), \label{eq:aj} \\
A_\Par &=& \sum_{1\le i\le A_\Net} \mathbf{1}(\ell_\Net[i] = P) + \sum_{1\le i\le A_\Ser} \mathbf{1}(\ell_\Ser[i] = P). \label{eq:aP} 
\end{eqnarray}
Moreover,
\begin{eqnarray}
	A_\sh &=& \sum_{1\le i\le A_1} \ell_1[i] + \sum_{1\le i\le A_2} \ell_2[i], \label{eq:ash}\\
	A_\Ser &=& \sum_{1\le i\le A_\Net} \mathbf{1}(\ell_\Net[i] = S) + \sum_{1\le i\le A_\sh} \mathbf{1}(\ell_\sh[i] = S), \label{eq:aS} \\
	A_\CT &=& \sum_{1\le i\le A_\Net} \mathbf{1}(\ell_\Net[i] = H) + \sum_{1\le i\le A_\Ser} \mathbf{1}(\ell_\Ser[i] = H) 
+ \sum_{1\le i\le A_\sh} \mathbf{1}(\ell_\sh[i] = H). \label{eq:aT}
\end{eqnarray}
\end{lemma}

\section{Cores In Random Networks}
\input{main.tex}

\section{Proof of Theorems~\ref{Main:Subcritical} and~\ref{Main:Critical}} \label{Sec:Proofs}

In this section we perform the proof of our main results. We shall denote throughout by $\Nl_n$ a network drawn uniformly at random from $\CN_n$ and by $\Nl$ a graph generated by $\Gamma N (\rho_N,1)$. 
\subsection{Small Cores}
As a first application of Corollary~\ref{cor:a_T} we prove that the counts of ``small'' cores in $\Nl_n$ are sharply concentrated around a specific value. Before we proceed let us make an auxiliary observation that will be used several times. Let $\CN(\CT)$ be an $\alpha$-nice class of networks. Let $\CB_n \subseteq \CN_n$ be some (bad) property of networks, and let $N\in\CN_n$. Moreover, let $A_\CT = A_\CT(N)$ be the total number of cores in $N$. As the random network $\Nl$ has the same chance of being any $N\in\CN_n$ we obtain by applying Corollary~\ref{cor:a_T} that there is a $C'>0$ such that
\[
	\BP(\Nl_n \in \CB_n)
	=
	\BP(\Nl \in \CB_n ~|~\Nl \in \CN_n)
	\le
	\BP(\Nl \in \CB_n, \, A_{\CT} \in (1\pm \eps/2)  a_{\CT} n ~|~\Nl \in \CN_n) + e^{-C'\eps^2 n},
\]
where $a_\CT = -2\frac{\rho_N'(1)}{\rho_N(1)}T(\rho_N, N_0)$. Applying Lemma~\ref{cor:GammaNIsEfficient}, we deduce that 
\begin{equation}
\label{eq:intermediate}
	\BP(\Nl_n \in \CB_n)
	\le
	O(n^{\beta})\cdot \BP(\Nl \in \CB_n, \, A_{\CT} \in (1\pm \eps/2)  a_{\CT} n ) + e^{-C'\eps^2 n}.
\end{equation}
The parameter $\beta$ is as in Lemma~\ref{cor:GammaNIsEfficient}. 
In particular, $\beta = 5/2$ if $\Phi_z(\rho_N,1,N_0)>0$, and $\beta = \alpha$ otherwise. This is a fact that we shall use below several times.
We start by counting cores of a given size.
\begin{lemma} \label{lem:counts}
Let $\CN(\CT)$ be an $\alpha$-nice class of networks for some $\alpha\in\mathbb{R}\setminus\{0\}$, and let $0 < \eps < 1$. Let $a_\CT$ be defined as in Corollary~\ref{cor:a_T}, $\beta$ as in Lemma~\ref{cor:GammaNIsEfficient}, and define for $k\ge 3$ the quantities
\begin{equation}
\label{eq:p_k}
	p_k := \frac{[x^{k-2}]T(x, N_0) \cdot \rho_N^{k-2}}{T(\rho_N, N_0)}
	\text{ and }
	k_0 := \max\{\ell ~|~ p_\ell n \geq 17\eps^{-2}a_\CT^{-1}(\beta + 1)\log n\}.
\end{equation}
Then there exists an constant $C= C(\CN) > 0$ such that for large $n$ and $k\leq k_0$  
$$ \prob (  c(k;\Nl_n)  \in \left(1\pm \eps \right) a_{\CT} p_k n ) \geq 1 -  e^{- C{\eps^2 p_k n}}.$$
\end{lemma}
\begin{proof}
Let $\CB_n \subseteq \CN_n$ be the set of networks whose number of cores with precisely $k$ vertices is not in $\left(1\pm \eps \right) a_{\CT} p_k n$. By applying~\eqref{eq:intermediate} we infer that it sufficient to show that, say, $\BP(\Nl \in \CB_n, \, A_{\CT} \in (1\pm \eps/2)  a_{\CT} n ) \le e^{-C'\eps^2p_k n}$.

Note that Lemma~\ref{lem:countk} asserts that $c(k;\, \Nl) = \sum_{i=1}^{A_\CT} \mathbf{1}(v(\ell_\CT[i])=k-2)$. As every $T\in \CT$ has $k$ vertices iff it has $k-2$ labeled vertices,  we infer also that $p_k = \BP(v(\ell_\CT[i]) =k-2)$. So,
\[
	\BP\left(\Nl \in \CB_n, \, A_{\CT} \in \left(1\pm \frac{\eps}2\right)  a_{\CT} n\right)
	\le
	\BP\left(\exists L\in \left(1\pm \frac\eps2\right)  a_{\CT} n : \sum_{i=1}^{L} \mathbf{1}(v(\ell_\CT[i]) = k-2) \not\in (1\pm\eps)a_\CT p_k n\right).
\]
By the Chernoff bounds and a union bound this probability is, say, at most $2ne^{-\frac{1}{16}\eps^2a_\CT p_kn}$. The proof then completes with~\eqref{eq:intermediate} and the choice of $k_0$ for large $n$.
\end{proof}
With the above Lemma at hand we are ready to prove the first statement in Theorem~\ref{Main:Subcritical} and the third statement in Theorem~\ref{Main:Critical}. Note that the analytic properties of $T$ asserted in Definition~\ref{Nice} imply with the Transfer Theorem (see e.g. Corollary VI.1 in~\cite{FlajSed}) that
\begin{equation}
\label{eq:xk2TxN0}
	[x^{k-2}]T(x, N_0) \sim \frac{t_{\alpha m}(N_0)}{\Gamma(\alpha)}k^{-\alpha-1}\rho_T(N_0)^{-k+2}.
\end{equation}
\begin{proof}[Proof of Theorem~\ref{Main:Subcritical}, (i)]
The condition $\Phi_z(\rho_N,1,N_0)>0$ together with Lemma~\ref{lem:branchPoint} imply that $\rho_N < \rho_T(N_0)$. Moreover, the definition of $p_k$ in~\eqref{eq:p_k} and~\eqref{eq:xk2TxN0} assert that there is a $C>0$ such that
\begin{equation}
\label{eq:pk_subcritical_asympt}
	p_k \sim Ck^{-\alpha-1}\left(\frac{\rho_N}{\rho_T(N_0)}\right)^k = Ck^{-\alpha-1}\tau^k.
\end{equation}
As $\tau < 1$, we infer that we can apply Lemma~\ref{lem:counts} for $k_0 = (1 - \delta)\log_{1/\tau}n$ whenever $n$ is sufficiently large. This completes the proof.
\end{proof}
\begin{proof}[Proof of Theorem~\ref{Main:Critical}, (iii)]
Here, the condition $\Phi_z(\rho_N,1,N_0)<0$ together with Lemma~\ref{lem:PhiSingular} imply that $\rho_N = \rho_T(N_0)$. The definition of $p_k$ in~\eqref{eq:p_k} and~\eqref{eq:xk2TxN0} assert that there is a $C>0$ such that
\begin{equation}
\label{eq:pk_asympototic}
	p_k \sim Ck^{-\alpha-1}\left(\frac{\rho_N}{\rho_T(N_0)}\right)^k = Ck^{-\alpha-1}.
\end{equation}
We infer that we can apply Lemma~\ref{lem:counts} for $k_0 = (\frac{n}{\omega(n) \log n})^{1/(\alpha+1)}$ whenever $n$ is sufficiently large.
\end{proof}
The next lemma deals with cores that contain more than $(\frac{n}{\omega(n) \log n})^{1/(\alpha+1)}$ vertices. The proof is essentially the same as the proof of Lemma~\ref{lem:counts} and hence omitted.
\begin{lemma} \label{lem:countsMany}
Let $\CN(\CT)$ be an $\alpha$-nice class of networks for some $\alpha\in\mathbb{R}\setminus\{0\}$, and let $0 < \eps < 1$. Let $a_\CT$ be defined as in Corollary~\ref{cor:a_T}, $\beta$ as in Lemma~\ref{cor:GammaNIsEfficient}, and define for $k\ge 3$ and $\xi \ge 1$
\begin{equation}
\label{eq:p_kxik}
	p_{k,\xi k} := \sum_{\ell = k}^{\xi k}\frac{[x^{\ell-2}]T(x, N_0) \cdot \rho_N^{\ell-2}}{T(\rho_N, N_0)}
	\text{ and }
	k_0 := \max\{\ell ~|~ p_{\ell, \xi\ell} n \geq 17\eps^{-2}a_\CT^{-1}(\beta + 1)\log n\}.
\end{equation}
Then there exists an constant $C= C(\CN) > 0$ such that for large $n$ and $k\leq k_0$  
$$ \prob (  c(k, \xi k;\,\Nl_n)  \in \left(1\pm \eps \right) a_{\CT} p_{k, \xi k} n ) \geq 1 -  e^{- C{\eps^2 p_{k, \xi k} n}}.$$
\end{lemma}
\begin{proof}[Proof of Theorem~\ref{Main:Critical}, (iv)]
The condition $\Phi_z(\rho_N,1,N_0)<0$ together with Lemma~\ref{lem:PhiSingular} imply that $\rho_N = \rho_T(N_0)$. The definition of $p_{k, \xi k}$ in~\eqref{eq:p_kxik} and~\eqref{eq:xk2TxN0} assert that for $\xi > 1$ there is a $C>0$ such that
\begin{equation}
\label{eq:pkxik}
	p_{k, \xi k} \sim \sum_{\ell = k}^{\xi k} C\ell^{-\alpha-1}
	\sim C k^{-\alpha}(1 - \xi^{-\alpha}).
\end{equation}
We infer that we can apply Lemma~\ref{lem:countsMany} for $k_0 = (\frac{n}{\omega(n) \log n})^{1/\alpha}$ whenever $n$ is sufficiently large.
\end{proof}
We now consider the case $\Phi_z(\rho_N,1,N_0)>0$. The next statement deals with the cases in Theorem~\ref{Main:Subcritical} not covered by Lemma~\ref{lem:counts}.
\begin{lemma}
Let $\CN(\CT)$ be an $\alpha$-nice class of networks for some $\alpha\in\mathbb{R}\setminus\{0\}$, and let $\eps > 0$. Assume that $\Phi_z (\rho_N(y),y, N_0(y))>0$, and let $\tau = {\rho_N \over \rho_T (N_0)}$. Then
$$ \prob \left( C_1 (\Nl_n) > (5/2 + \eps) \log_{1/\tau} n \right) = o(n^{-\eps}).$$
Moreover, we have 
$$ \prob \left( c \left( (1-\eps)\log_{1/\tau} n, {5\over 2} \log_{1/\tau} n;\, \Nl_n\right) > n^{2\eps} \right) =o(1).$$
\end{lemma}
\begin{proof}
Let $\CB_n \subset \CN_n$ be the set of networks in which the largest core has size $> (5/2 + \eps) \log_{1/\tau} n$. By applying~\eqref{eq:intermediate} we obtain
\[
	\BP(\Nl_n \in \CB_n)
	\le
	O(n^{5/2})\cdot \BP(\Nl \in \CB_n, \, A_{\CT} \in (1\pm \eps/2)  a_{\CT} n ) + e^{-C'\eps^2 n}.
\]
Note that Lemma~\ref{lem:countk} asserts that
$$
\Nl \in \CB_n \implies \sum_{i=1}^{A_\CT} \mathbf{1}\left(v(\ell_\CT[i])> (5/2 + \eps) \log_{1/\tau} n - 2\right) ~>~ 0.
$$
As $\BP(v(\ell_\CT[i]) = k) = \frac{[x^{k-2}]T(x, N_0)]\cdot \rho_N^{k-2}}{T(\rho_N, N_0)}$ we obtain by using~\eqref{eq:xk2TxN0} that there is a $C>0$ such that for large $n$
\[
	\BP\left(v(\ell_\CT[i])> (5/2 + \eps) \log_{1/\tau} n - 2\right)
	\le C \sum_{k > (5/2 + \eps) \log_{1/\tau} n - 2} k^{-\alpha-1}\tau^k.
\]
The condition $\Phi_z(\rho_N,1,N_0)>0$ together with Lemma~\ref{lem:branchPoint} imply that $\rho_N < \rho_T(N_0)$, and hence $\tau < 1$. Thus, $\BP\left(v(\ell_\CT[i])> (5/2 + \eps) \log_{1/\tau} n\right) = o(n^{-5/2-\eps})$. The proof of the first statement completes with Markov's inequality.

The second statement follows similarly by estimating the probability
\[
	\BP\left((1-\eps)\log_{1/\tau} n \le v(\ell_\CT[i])\le (5/2 + \eps) \log_{1/\tau} n\right)
\]
with the same technique as above. The straightforward details are left to the reader.
\end{proof}

\subsection{The Largest Core}

In this section we prove the first two statements in Theorem~\ref{Main:Critical}. In particular, we show the following.
\begin{lemma}\label{lem:uniquelarge}
Let $\CN(\CT)$ be an $\alpha$-nice class of networks for some $\alpha\in\mathbb{R}\setminus\{0\}$, and let $\eps > 0$. Assume that $\Phi_z (\rho_N(y),y, N_0(y))<0$, and let $\omega(n)$ be a function such that $\lim_{n\to\infty} \omega(n) = \infty$. Then, a.a.s.
$$\big| C_1(\Nl_n) -  \gamma_{\T}n  \big| < \eps n,$$
where 
$\gamma_{\T} := v_{\T} - \alpha_{\T} {\rho_N T_x (\rho_N, N_0)\over T(\rho_N, N_0)}$, and $v_\CT$ is as in Lemma~\ref{lem:system} and $a_\CT$ as in Corollary~\ref{cor:a_T}. Moreover, there are a.a.s. no other cores in $\Nl_n$ with more than $\omega(n) \cdot n^{1/\alpha}$ vertices.
\end{lemma}
\begin{proof} 
Set $n_0 := \omega(n) \cdot n^{1/\alpha}$. We will use a counting argument to show that the number of networks that have a unique core with more than $n_0$ vertices is asymptotically much larger than the number $A_2$ of networks that have at least two such cores.

The singularity expansion of $N(x,y)$ with respect to $x$ (Lemma~\ref{lem:PhiSingular}) 
together with the Transfer Theorem (Corollary~VI.1 in~\cite{FlajSed}) imply that there exists a constant $c$ such that as~$\ell \rightarrow \infty$
\begin{equation} \label{eq:NetsCount}
|\N_{\ell } | = (1+o(1)) c \rho_N^{-\ell} \ell^{-\alpha-1} \ell!.
\end{equation}

We bound $A_2$ as follows. Let $N$ be a network having $n$ labeled vertices. Assume that $N$ contains at least 
two cores each having at least $n_0$ vertices. Then there is either a cut-edge that splits $N$ in two networks that contain at least $n_0$ vertices each, or such a splitting can be obtained if we choose the two poles as the cut-set. So, every such network can be described by a triple $(N_1, e, N_2)$, where $e\in N_1$ (or, in slight abuse of notation, $e$ contains the poles of~$N_1$), $v(N_1), v(N_2) \ge n_0-2$, and we can construct $N$ by identifying the poles of $N_2$ with the endpoints of~$e$.


Note that by Theorem~\ref{thm:edges} we can assume that with probability $1-o(1)$ there are at most~$Kn$ choices for $e$, for some constant $K>0$.  Moreover, there are ${n \choose v(N_1)}$ ways to choose the labels for $N_1$.
By summing over all choices for $v(N_1)$ we obtain for $n$ large enough that
\begin{equation} \label{eq:Sum2}
\begin{split}
A_2 & \leq 2Kn \sum_{n_0/2 \leq s \leq n - n_0}{n \choose s}~|\N_s|~|\N_{n-s}|  + o(|\CN_n|)\\
&\leq 4K n\sum_{n_0/2 \leq s \leq n - n_0}{n \choose s} \cdot \rho_N^{-s} s^{-\alpha - 1} s! \cdot \rho_N^{-n+s} (n-s)^{-\alpha - 1} (n-s)! + o(|\CN_n|) \\
& = 4K n \rho_N^{-n}n!  \sum_{n_0/2 \leq s \leq n - n_0} s^{-\alpha - 1} (n-s)^{-\alpha-1 } + o(|\CN_n|)\\ 
&= 8K n  \rho_N^{-n}n! \sum_{n_0/2 < s \leq n/2} s^{-\alpha-1} (n-s)^{-\alpha-1} + o(|\CN_n|).
\end{split}
\end{equation}
The last sum can be bounded above as follows: 
\begin{equation*}
\begin{split}
& \sum_{n_0/2 \leq s \leq n/2} s^{-\alpha - 1} (n-s)^{-\alpha-1}
\leq \left(n \over 2\right)^{-\alpha-1}\sum_{n_0/2 \leq s \leq n/2} s^{-\alpha - 1} \
\leq  \left(n \over 2\right)^{-\alpha-1} \int_{n_0/2-1}^{n/2}s^{-\alpha - 1} ds \\
& = O \left(n^{-\alpha-1} n_0^{-\alpha}\right) =   O \left(n^{-\alpha-2} \omega(n)^{-1}\right).
\end{split}
\end{equation*}
So using (\ref{eq:NetsCount}) with $\ell=n$, we obtain $A_2 = o (|\N_n|)$.

This shows that there is a unique largest core with more than $n_0$ vertices. To complete the proof we show the claimed bound on its size. First, by applying Lemma~\ref{lem:system} we infer that a.a.s.\ the total number $\sum_k(k-2)c(k;\, \Nl_n)$ of labeled vertices in all cores of $\Nl_n$ is 
in $v_{\T} n \pm {\eps n \over 2}$. We now argue that the number total number $v_{n_0}$ of labeled vertices that lie in cores with a most $n_0$ labeled vertices is a.a.s.\ in $a_\CT\frac{\rho_N T_x(\rho_N, N_0)}{T(\rho_N, N_0)} \pm \frac{\eps n}2$. This will conclude the proof.

First note that $ v_{n_0} = \sum_{k=2}^{n_0} (k-2) c(k;\Nl_n)$.
Let $n_- := \left( {n \over \omega(n) \log n}\right)^{1/(\alpha +1)}$. 
We obtain a lower bound on $v_{n_0}$ by keeping the summands up to~$n_-$ in the above sum: 
 $v_{n_0} \geq  \sum_{k=2}^{n_-} (k-2) c(k;\Nl_n). $
Lemma~\ref{lem:counts} yields that for all~$3\le k \le n_-$ we have 
$c(k;\Nl_n) \geq (1-\eps) p_k a_{\T} n$ with probability at least $1 - \sum_{k=2}^{n_-} e^{- \Omega(\eps^2 p_k n)} = 1-o(1)$. 
So, for large $n$
\begin{equation}
\label{eq:Lower} 
\begin{split}
v_{n_0} &\geq (1-\eps) a_{\T} n \sum_{k=2}^{n_-} (k-2) p_k = (1-\eps) a_{\T} n 
\sum_{k=2}^{n_-} (k-2) {\rho_N^{k-2} [x^{k-2}] T(x,N_0 ) \over T(\rho_N,N_0 )} \\
& \geq  (1 -\eps) a_{\T} n {\rho_N T_x(\rho_N,N_0) \over T(\rho_N,N_0)} - {\eps n \over 2}, 
\end{split}
\end{equation}
where the last inequality follows from the fact that the sum $\sum_{k=2}^{\infty} k {\rho_N^k [x^k] T(\rho_N,N_0)}$ is convergent. 
Using the same reasoning, we also deduce that a.a.s. 
\begin{equation} \label{eq:FirstSumUpper} 
\begin{split} 
 \sum_{k=2}^{n_-} & (k-2) c(k;\Nl_n) \leq (1+\eps) a_{\T} n \sum_{k=2}^{n_-} (k-2) {\rho_N^{k-2} [x^{k-2}] T(x,N_0) \over T(\rho_N,N_0)} \\ 
& \leq (1+\eps) a_{\T} n {\rho_N T_x(\rho_N,N_0) \over T(\rho_N,N_0)} + \eps n. 
\end{split}
\end{equation}
Thus, to deduce a matching upper bound for $v_{n_0}$ it is sufficient to show that a.a.s.
\begin{equation} \label{eq:SecondSumToProve} 
\sum_{k=n_-}^{n_0} (k-2) c(k;\Nl_n) \leq  \eps n. 
\end{equation}
Let us set $n_+ :=\left( {n \over \omega(n) \log n}\right)^{1/\alpha}$. Moreover, set $\xi = (\omega(n)^{\alpha+1} \log n)^{1/\alpha}$ and note that the preconditions of Lemma~\ref{lem:countsMany} are satisfied with $k = n_+$ and this choice for $\xi$. We infer that a.a.s.\ for all $3 \le k \le n_+$
\[
	c(k, \xi k;\, \Nl_n) \le \frac32 \cdot p_{k, \xi k} n,
	~\text{ where }~
	p_{k, \xi k} = \sum_{\ell = k}^{\xi k} \frac{[x^{\ell-2}]T(x, N_0) \cdot \rho_N^{\ell-2}}{T(\rho_N, N_0)}
\]
Thus, by using~\eqref{eq:pk_asympototic} we infer that there is a $C'>0$ such that
\[
	\sum_{k= n_-}^{n_0} (k-2) c(k;\Nl_n)
	\leq 
	C\xi n \sum_{k= n_-}^{n_0} (k-2) \frac{[x^{k-2}]T(x, N_0) \cdot \rho_N^{\ell-2}}{T(\rho_N, N_0)} = o(n).
\]
This shows~\eqref{eq:SecondSumToProve}, and the proof is completed.
\end{proof}

\section{Remaining Proofs} \label{Sec:remProofs}
\input{proofs}

\medskip

{\footnotesize \obeylines \parindent=0pt




\end{document}

%% file: preliminaries.tex
\subsection{Boltzmann Sampling}
\label{ssec:boltzmannSampling}
The above decomposition of networks can be expressed in terms of certain combinatorial operators, which are omnipresent in modern theories of enumeration as well as random generation of combinatorial structures. We refer the reader to the book by Flajolet and Sedgewick~\cite{FlajSed} for a detailed exposition of these techniques.

Duchon et al.~\cite{Duch} have shown that the above operations can be used to design \emph{Boltzmann samplers}. If $\gamma$ is a graph, we denote by $e(\gamma)$ and $v(\gamma)$ the number of edges and the number of labeled vertices of $\gamma$, respectively. A Boltzmann sampler for $\G$ is a randomized algorithm  that generates any $\gamma \in \G$ with probability 
${1\over G(x,y)}~{x^{v(\gamma)} y^{e(\gamma)}\over v(\gamma)!}$.
Note that if we set $y=1$ and we condition on the order of the output graph being $n$, then the distribution of the output graphs is the uniform distribution on the set of graphs of order $n$. 

\subsection{Networks}
The decomposition of networks in Section~\ref{Sec:Nets} implies directly the following statement.
\begin{lemma}
\label{lem:decomposition}
Let $e$ denote the class of networks that consist of precisely one edge and $\mathcal{X}$ denote the class which consists of the empty graph 
with one labeled vertex. Moreover, let~$\CN$ be an $\alpha$-nice class of networks with 3-cores from the class $\CT$. Then the classes~$\CN, \CS, \CP, \CH$ satisfy the following relations.
\[
\begin{split}
	\CN & = e + \CS + \CP + \CH, \\
	\CS & = (e + \CP + \CH) \times \mathcal{X} \times \CN, \\
	\CP & = e\times\Set_{\ge 1}(\CS + \CH) + \Set_{\ge 2}(\CS + \CH),\\
	\CH & = \CT \circ_e \CN.
\end{split}
\]
\end{lemma}
Using this fact we immediately obtain the following relation for the corresponding exponential generating functions. This result was shown, among others, in~\cite{Traht, Walsh}.
\begin{lemma}
\label{lem:egfsNetworks}
The exponential generating functions enumerating networks satisfy
\begin{equation} 
\begin{split}
N(x,y) &= y + S(x,y) + P(x,y) + H(x,y), \\
S(x,y) &= (y + P(x,y) + H(x,y))xN(x,y), \\ 
P(x,y) &= (1+y)\left(e^{S(x,y) + H(x,y)}-1\right) -S(x,y) - H(x,y),\\
\label{eq:4th} H(x,y) &= T(x,N(x,y)).
\end{split}
\end{equation} 
Moreover, $N(x,y)$ satisfies the equation $\Phi(x, y, N(x,y)) = 0$, where $\Phi$ is as in~\eqref{eq:Phi}.
\end{lemma}

%% file: main.tex
\label{sec:CoresInRandomNets}
In our proofs we will be using repeatedly the Chernoff bound on the probability that a binomially distributed random 
variable deviates significantly from its expected value. We will use the version which appears in~\cite{JLR}. 
In particular, let $X$ be a binomially distributed random variable and let $t>0$.  Then 
\begin{equation} \label{eq:Chernoff}
\prob (|X - \ex(X)| > t) \leq 2 \exp \left(- {t^2 \over 2(\ex(X) + t/3)} \right).
\end{equation}
Another technical ingredient in our proof is the following lemma. It gives a Chernoff-type bound for the sum of independent $\mathtt{Po}_{\ge j}(\mu)$
variables, and its proof uses exponential generating functions. For the sake of completeness, it can be found  in Section~\ref{ssec:proofsSystem}.
\begin{lemma}
\label{lem:poissonConditional}
Let $X_1, \dots, X_r$ be independent $\mathtt{Po}_{\ge j}(\mu)$ variables, where $\mu >0$ and $j \in \{1,2\}$. Let $Y_{r'} = \sum_{1 \le i \le r'}X_i$, where $0\le r'\le r$. Then, there is a $C>0$ such that for any $\frac{\log r}{\sqrt{r}}\le \eps \le 1$ and sufficiently large $r$
\[
	\Pr{\exists \ 0\le r'\le r~:~\left|Y_{r'} - \E{Y_{r'}}\right| \ge \eps r} \le e^{-C\eps^2r}.
\]
\end{lemma}
In the following we denote by~$A_s$ the random variable counting the number of values from~$\ell_s$ used by an execution~$\Gamma N(\rho_N (y),y)$, where~$s \in \{\Net, \Ser, \Par, \sh, 1, 2, \CT\}$. Moreover,we write~$V_\CT$ and~$E_\CT$ for the total number of labeled vertices and edges in all cores of~$\Gamma N(\rho_N (y),y)$, i.e.,
\[
	V_\CT = \sum_{i=1}^{A_\CT} v(\ell_\CT[i]) = \sum_k(k-2)c(k;\, \Gamma N(\rho_N (y),y)),
	\quad\text{ and }\quad
	E_\CT = \sum_{i=1}^{A_\CT} e(\ell_\CT[i]).
\]
Finally, let us denote by~$\Nl$ the output of the Boltzmann sampler~$\Gamma N (\rho_N,1)$. 
\newcommand{\balpha}{\mathbf{\alpha}}
\begin{lemma} \label{lem:system}
Let~$0 < \eps <1$. There is a constant~$C>0$ such that for sufficiently large~$n$ and any~$Z\in\{A_\Net, A_\Ser, A_\Par,  V_\CT, E_\CT\}$
\[
	\Pr{|Z - zn|\le \eps n ~|~ \Nl \in \CN_n} \ge 1 - e^{-C\eps^2 n},
\]
and~$z \in \{a_\Net, a_\Ser, a_\Par,  v_\CT, e_\CT \}$, where 
$\balpha = [a_\Net, a_\Ser,  a_\Par,  v_\CT, e_\CT ]^T$ is the unique solution of the linear system~$M\alpha = r$, with
\begin{equation}
\label{eq:system}
	M
	=
\left[ \begin{array}{ccccc} 
{1\over N_0} & {\rho_N N_0 \over S_0} & {N_0-1\over 2P_0} & 0 & 0 \\ 
0 & 1 & 0 & 1 & 0 \\ 
{S_0\over N_0}& -1 & {S_0N_0 \over P_0} & 0 & 0 \\ 
{P_0 \over N_0} & {\rho_NP_0N_0 \over S_0} & -1 & 0  & 0\\ 
-1 & 1 & 0 & 0 & 1 
\end{array} \right],
	r
	=
	\left[ \begin{array}{c} \mu \\ 1 \\ 0 \\ 0 \\ 0 \end{array} \right]
	~\text{, and }~
	\mu = -\frac{\rho'_N(1)}{\rho_N(1)}.
\end{equation}
\end{lemma}
\begin{proof}
First, by applying Lemma~\ref{lem:relations}, Statements \eqref{eq:nSvT} and \eqref{eq:aNaSeT}, we obtain that for every point in the conditional probability space ``$\Nl \in \CN_n$''
\[
	n = A_\Ser + V_\CT
	\enspace \text { and }\enspace
	A_\Net = A_\Ser + E_\CT,
\]
from which we immediately obtain that
$
	1 = \frac{A_\Ser}n + \frac{V_\CT}n
$
and
$
	\frac{A_\Net}n = \frac{A_\Ser}n + \frac{E_\CT}n.
$
These two facts are the second and the last line of the linear system above, and thus hold with probability~1. In the remainder we argue that all other equations are true with probability at least $1 - e^{-C\eps^2n}$. This completes the proof of the lemma with the following reasoning. The determinant of $M$ equals
\[
	D = \rho_NN_0^2+(\rho_N+1)N_0+1.
\]
Now, since $N(x)$ has only non-negative coefficients, we infer that $D \neq 0$ and the proof is finished.

Let us continue with an auxiliary observation. Consider e.g.\ Statement \eqref{eq:aj} in Lemma~\ref{lem:relations}. Our aim is to translate this statement into a~\emph{high probability statement} for the relation of the random variables $A_j$ and $A_\Par$. For this, let $p_j = \Pr{\Par(\rho_N,1) = j}$ and note by applying Lemma~\ref{cor:GammaNIsEfficient} we infer that there is a $c>0$ such that for large $n$
\[
	\Pr{A_j \not\in (1\pm\eps)p_jA_\Par \pm \eps n ~|~ \Nl \in \CN_n}
	\le cn^\beta \Pr{A_j \not\in (1\pm\eps)p_jA_\Par \pm \eps n}.
\]
By applying \eqref{eq:aj} we obtain $A_j = \sum_{i=1}^{A_\Par}\mathbf{1}(\ell_\Par[i] = j)$. Using this, we infer that the probability of the above event is at most
\[
	cn^\alpha\Pr{\exists L\ge 1: ~ \sum_{i=1}^{L}\mathbf{1}(\ell_\Par[i] = j) \not\in (1\pm\eps)p_j L \pm \eps n}
\]
and a simple union bound together with the Chernoff bounds imply that there is a $C>0$ such that
\[
	\Pr{A_j \not\in (1\pm\eps)p_jA_\Par \pm \eps n ~|~ \Nl \in \CN_n} \le e^{-C\eps^2n}.
\]
In other words, we have demonstrated that~\eqref{eq:aj} translates into 
the statement
\begin{equation}
\label{eq:aj_conc}
	A_j \in (1\pm\eps)p_jA_\Par \pm \eps n
\end{equation}
with probability $\ge 1 - e^{-C\eps^2n}$. Now, with exactly the same line of reasoning we can infer from~\eqref{eq:edges} and~\eqref{eq:aP} that if we condition on ``$\Nl \in \CN_n$'', with probability at least $1 - e^{-C\eps^2n}$
\begin{eqnarray}
e(\Nl) & \in & (1\pm\eps)\left(\frac1{N_0}A_\Net + \frac{\rho_N N_0}{S_0}A_\Ser + \frac{e^{S_0+H_0}-1}{P_0}A_\Par\right) \pm\eps n, \label{eq:edges_conc} \\
A_\Par & \in &  (1\pm\eps)\left(\frac{P_0}{N_0}A_\Net + \frac{\rho_NP_0N_0}{S_0}A_\Ser\right) \pm\eps n. \label{eq:aP_conc}
\end{eqnarray}
Moreover, by applying Lemma~\ref{lem:poissonConditional} instead of the Chernoff bounds we infer from~\eqref{eq:ash} that with probability at least $1 - e^{-C\eps^2n}$
\begin{equation}
A_\sh \in (1\pm\eps)\left(\frac{(S_0+H_0)e^{S_0+H_0}}{e^{S_0+H_0}-1}A_1 + \frac{(S_0+H_0)(e^{S_0+H_0}-1)}{e^{S_0+H_0} - 1 - S_0-H_0}A_2\right) \pm\eps n. \label{eq:ash_conc}
\end{equation}
Finally, from~\eqref{eq:aS} and~\eqref{eq:aT} we infer again by the Chernoff bounds that with probability at least $1 - e^{-C\eps^2n}$
\begin{eqnarray}
A_\Ser & \in &  (1\pm\eps)\left(\frac{S_0}{N_0}A_\Net + \frac{S_0}{S_0+H_0}A_\sh\right) \pm\eps n, \label{eq:aS_conc} \\
A_\CT & \in &  (1\pm\eps)\left(\frac{H_0}{N_0}A_\Net + \frac{\rho_NH_0N_0}{S_0}A_\Ser + \frac{H_0}{S_0+H_0}A_\sh\right) \pm\eps n. \label{eq:aT_conc}
\end{eqnarray}
Moreover, Theorem~\ref{thm:edges} implies that there is a $B>0$ such that
\[
	\Pr{e(\Nl) \in (1\pm\eps)\mu n ~|~ \Nl \in \CN_n} \ge 1 - e^{-B\eps^2 n}.
\]
Let $\vec{X} = \frac1n[A_\Net, A_\Ser, A_\Par,  V_\CT, E_\CT]^T$. By combining all the above facts together, and using the fact $N(x,y) = y + (1+y)(e^{S(x,y)+H(x,y)}-1)$, which follows from Lemma~\ref{lem:egfsNetworks}, we infer that with probability at least $1 - e^{-C''\eps^2n}$ there is a $c>0$ such that
\[
	(1+\eps) M \cdot \vec{X} \ge r - \eps \cdot \vec{c}
	\quad \text{ and }\quad
	r + \eps \cdot \vec{c} \ge (1-\eps) M \cdot \vec{X}
\]
where $\vec{c} = [c,c,c,c,c]^T$. The proof completes by elementary linear algebra algebra, and by choosing the $\eps$ in this proof as $\eps/c'$, for a suitable $c'>0$.
\end{proof}
Note that the above statement does not yield any information about the total number of cores in a random network with $n$ vertices. However, with a little additional work we arrive at the following result.
\begin{corollary} \label{cor:a_T} 
Let $0< \eps< 1$. Then, there is a $C >0$ such that
\[
	\BP[|A_\CT - a_\CT n| \ge \eps a_\CT n ~|~ \Nl \in \CN_n] \le e^{-C\eps^2 a_\CT n},
\]
where $a_\CT = 2\mu T(\rho_N, N_0)$ and $\mu = -\frac{\rho_N'(1)}{\rho_N(1)}$.
\end{corollary}
\begin{proof} 
By solving~\eqref{eq:system} we obtain explicit (but lengthy) expressions for the highly probable values of $A_\Net$, $A_\Ser$, and $A_\sh$. Then, by using~\eqref{eq:aT_conc} we obtain after elementary algebraic manipulations the claimed statement.
\end{proof}


%% file: proofs.tex
\subsection{Proofs of Section~\ref{Sec:Nets}}
\label{ssec:proofNets}

\begin{proof}[Proof of Proposition~\ref{prop:transferNB}]
Firstly, let us observe that there is a natural projection of $(\CN + 1)\times \X^2$ onto $\CN$ which maps an ordered pair in 
 $(\CN + 1)\times \X^2$ to its first element, which we call the \emph{underlying network}. 
Note also that for each network in $\CN_{n-2}$ its preimage 
under this map consists of $n(n-1)$ elements of $(\CN +1 )\times \X^2$. Therefore, since $\Pa$ is a property that is closed 
under automorphisms the proportion of networks in $\CN_{n-2}$ that have $\Pa$ equals the proportion of elements of 
$(\CN + 1)\times \X^2$ whose underlying network has $\Pa$. In other words, the probability that $\Nl_{n-2} \in \Pa$ equals the 
probability that if we choose uniformly at random an element of $(\CN + 1)\times \X^2$ then the underlying network has $\Pa$. 
We denote by $P$ the subset of $(\CN + 1)\times \X^2$ which is the preimage of the networks in $\N_{n-2}$ which have the property 
$\Pa$. So the proportion of $P$ in $(\CN + 1)\times \X^2$ is also at least $1-f(n-2)$. 

The latter probability space is mapped bijectively into $(1+e) \times \vec{\CB}_n$. 
So, in particular, $P$ is mapped bijectively into a set 
$P' \subseteq (1+e) \times \vec{\CB}_n$. Let us consider the subspace $e \times \vec{\CB}_n$ which 
contains precisely half of the elements of $(1+e)\times \vec{\CB}_n$. We have 
$${|P' \cap e \times \vec{\CB}_n |\over |e \times \vec{\CB}_n|} \geq 1 - 2f(n-2).$$    

On the other hand, the subset $e \times \vec{\CB}_n$ is mapped bijectively into the set 
$\vec{\CB}_n^{+} := \{ (B,e) \ : \ B \in \CB_n, \ e \in E(B)\}$. Therefore, if $P''$ is the image of $P'$ under this isomorphism, then 
${|P''| \over |\vec{\CB}_n^{+}|} \geq 1 -2f(n-2)$.  

But now we are able to relate probabilities in the uniform space $\vec{\CB}_n^+$ with probabilities in the space $\CB_n$. 
Note that there is a natural projection of  $\vec{\CB}_n^+$ onto $\CB_n$ where a pair $(B,e) \in \vec{\CB}_n^+$ is mapped 
to $B \in \CB_n$; we denote this by $\pi$.  
So $|\vec{\CB}_n^+| \leq \kappa n |\CB_n|$, as every graph in $\CB_n$ has at most $\kappa n$ edges. Furthermore, let $\overline{P}''$ be the complement of $P''$ in $\vec{\CB}_n^+$. Since 
every graph in $\CB_n$ contains at least $n$ edges, we have $|\pi (\overline{P}'')|\leq |\overline{P}''|/n$. 
Therefore we arrive at the relation
${|\pi (\overline{P}'') |\over |\CB_n|} \leq \kappa {|\overline{P}''|\over |\vec{\CB}_n^+|}$. Since the latter ratio is at most $2f(n-2)$, 
it turns out that ${|\pi (\overline{P}'') |\over |\CB_n|} \leq 2 \kappa f(n-2)$.  
\end{proof}  

\subsection{Proofs of Section~\ref{sec:singAnalysis}}

\begin{proof}[Proof of Theorem~\ref{thm:singTransferGeneral}]
Let $s = \lfloor k/m\rfloor$ and abbreviate $F = \left(1 - {f}/{r(x,y)}\right)^{{1}/{m}}$. The analyticity of $g,h$ implies that $G$ can be represented as
\[
	G(x,y,f) = a_0 + a_mF^m + \dots + a_{sm}F^{sm} + a_kF^k + \dots,
\]
where the $a_i = a_i(x,y)$ are auxiliary analytic functions. In particular, we have that
\begin{equation}
\label{eq:a0amak}
	a_0 = g(x,y,r(x_0, y_0)),
	\enspace
	a_m = -g_f(x,y,r(x_0, y_0))r(x_0, y_0),
	\enspace \text{and} \enspace
	a_k(x_0, y_0) \neq 0.
\end{equation}
The definition of $F$ implies that $f = r(1 - F^m)$, where $r = r(x,y)$. The (unknown) function~$f$ hence satisfies the equation
\begin{equation}
\label{eq:vorbereitung}
	r - a_0 = (a_m + r)F^m + a_{2m}F^{2m} + \dots + a_{sm}F^{sm} + a_kF^k + \dots.
\end{equation}
Note that $a_0(x_0, y_0) = g(x_0, y_0, r(x_0, y_0)) = r(x_0, y_0)$, which implies that the the left-hand side of the above equation vanishes at $(x_0, y_0)$. Moreover, our assumptions imply that $r_x(x_0, y_0) - (a_0)_x(x_0, y_0) = r_x(x_0, y_0) -  g_x(x_0, y_0, r(x_0, y_0)) \neq 0$. By applying the Preparation Theorem by Weierstrass we thus infer the existence of analytic functions $H(x,y)$ and $\rho(y)$ such that $H(x_0, y_0) \neq 0$, $\rho(y_0) = x_0$ and locally around $(x_0, y_0)$
\[
	r - a_0 = H(x,y)(x - \rho(y)).
\]
Set $X = (1 - x/\rho(y))^{1/m}$. Equation~\eqref{eq:vorbereitung} is then around $(x_0, y_0)$ equivalent to
\[
	(-H(x,y)\rho(y))X^m = F^m\left((a_m + r) + a_{2m}F^{m} + \dots + a_{sm}F^{(s-1)m} + a_kF^{k-m} + \dots\right).
\]
Recall that for any $|x|<1$ and $\alpha\in\mathbb{C}$ we have that $(1 + x)^\alpha = \sum_{k\ge 0} \binom{\alpha}{k}x^k$. Using this, the above is equivalent to
\[
	(-H(x,y)\rho(y))^{1/m}X = F\left((a_m + r)^{1/m} + \tilde{a}_{2m}F^{m} + \dots + \tilde{a}_{sm}F^{(s-1)m} + \tilde{a}_kF^{k-m} + \dots\right), 
\]
where the $\tilde{a}_i$'s are functions given in terms of the $a_i$'s, and in particular $\tilde{a}_k = \frac{a_k}{m(a_m + r)^{{(m-1)}/{m}}}$. As $H(x_0,y_0)\rho(y_0) \neq 0$ and 
\[
	a_m(x_0,y_0) + r(x_0,y_0) \stackrel{\eqref{eq:a0amak}}{=} (1 - g_f(x_0,y_0,f_0))r(x_0,y_0) \neq 0,
\]
the above relation between $X$ and $F$ is locally invertible around $(x_0,y_0)$. By indeterminate coefficients we obtain
\[
	F = \left(\frac{-H(x,y)\rho(y)}{a_m+r}\right)^{1/m} X + b_{m+1}X^{m+1} + \dots -\tilde{a}_k\frac{(-H(x,y)\rho(y))^{\frac{k-m+1}{m}}}{(a_m+r)^{\frac{k-m+2}{m}}}X^{k-m+1} + \dots,
\]
where the $b_i$'s, $i \in\{ jm + 1 : 1\le j\le s\}$ are analytic functions of the $\tilde{a}_i$'s. By taking the $m$th power of both sides of the above equation we obtain
\[
	1 - \frac{f}{r(x,y)} = \frac{-H(x,y)\rho(y)}{a_m+r}X^m + \dots - \tilde{a}_km\frac{(-H(x,y)\rho(y))^{k/m}}{(a_m+r)^{(k+1)/m}}X^k + \dots.
\]
This completes the proof of~\eqref{eq:fsingExp}, as it readily follows from our assumptions that the coefficient of $X^k$ above is $\neq 0$ in a neighborhood of $(x_0, y_0)$.
\end{proof}

\subsection{Proofs of Sections~\ref{Sec:Samplers} and~\ref{sec:CoresInRandomNets}}
\label{ssec:proofsSystem}

\begin{proof}[Proof of Lemma~\ref{lem:relations}]
To see (\ref{eq:nSvT}), we define a mapping from the set of vertices of the resulting network into  the set of calls of the 
samplers $\Gamma N, \Gamma S, \Gamma P$ and $\Gamma H$. More specifically, we map a vertex to the call of the subroutine where 
this vertex appeared for the first time. Note that, in fact, the image of this map is the union of the calls of $\Gamma S$ and $\Gamma H$ only,
as these are the only subroutines where vertices are created. The preimage of each call of $\Gamma S$ consists of only one vertex, whereas 
the preimage of each call of $\Gamma H$ consists of as many vertices as the number of vertices of the sample from $\CT$, which is used in the 
call of $\Gamma H$. Note that once a vertex has been created it can never be identified with another vertex but only with a pole.  Thus 
(\ref{eq:nSvT}) follows. 

Similarly, (\ref{eq:edges}) follows from a similar mapping of the set of edges of the resulting network to the set of calls of the samplers 
$\Gamma N, \Gamma S, \Gamma P$ and $\Gamma H$, where an edge is mapped to that call where it was created. In this case, the image 
of the mapping is contained into the union of the sets of calls of $\Gamma N$, $\Gamma S$ and $\Gamma P$. This is the case, since 
in any call of $\Gamma H$ each edge of the sample from $\T$ is replaced by a network which is the result of a call of $\Gamma N$. 
and it is the calls of $\Gamma N$ which may yield an edge. Also, among the calls of $\Gamma S$ those 
which create an edge are those whose left part consists of an edge. Finally, the calls of $\Gamma P$ which create an edge are those which 
result into a parallel network  of the first type, that is, a parallel network which consists of an edge and a set of at least one $\CS $ or $\CH$ 
network. Thus, (\ref{eq:edges}) follows. 

Equation (\ref{eq:aNaSeT}) follows by mapping the set of calls of $\Gamma N$ (apart from the initial call) 
to the union of the sets of calls of  $\Gamma S$ or $\Gamma H$, where a call of $\Gamma N$ is mapped to the subroutine which started it.
(Note that, apart from the initial call of $\Gamma N$, it is only these samplers that are able to call $\Gamma N$.)  In particular, 
a call of $\Gamma H$ makes precisely one call of $\Gamma T$ and, conversely, each call of $\Gamma T$ is made by a call of $\Gamma H$. 
Furthermore, for each edge of the sampled network of type $\T$ precisely one call of $\Gamma N$ is made. This concludes the proof of (\ref{eq:aNaSeT}). 

As far as (\ref{eq:aj}) is concerned this is simply counting the number of calls of $\Gamma P$ which yield a parallel network of the first type 
(if $j=1$) or a parallel network of the second type (if $j=2$). 

Equation (\ref{eq:aP}) follows with a similar argument by mapping each call of $\Gamma P$ to the call of $\Gamma N$ or $\Gamma S$ which 
created it, as these are the only types among our Boltzmann samplers which can call directly $\Gamma P$. Thus (\ref{eq:aP}) holds and the 
same kind of argument works also for (\ref{eq:aS}) as well as for (\ref{eq:aT}). For the latter, we need the observation that there is a
one-to-one correspondence between the calls of $\Gamma H$ and the calls of $\Gamma T$.
\end{proof}

\begin{proof}[Proof of Lemma~\ref{lem:poissonConditional}]
We start with the case $j = 1$. For a $\mathtt{Po}_{\ge1}(\mu)$ variable $X$ and $\xi \ge 1$ we have $\E{\xi^{X}} = \frac{e^{\xi\mu}-1}{e^\mu-1}$. Using Markov's inequality, this implies for any $\xi \ge 1$, $0 \le r'\le r$, and~$\eps > 0$
\begin{equation}
\label{eq:f}
	\Pr{Y_{r'} \ge (1 + \eps)r'\E{X}}
	\le 
	e^{r'f(\xi)}
	, \text{ where }
	f(\xi) = \log\left(\frac{e^{\xi\mu} - 1}{e^\mu-1}\right) -(1 + \eps) \E{X} \log\xi.
\end{equation}
Note that
\[
	f'(\xi) = \frac{\mu e^{\xi\mu}}{e^{\xi\mu} - 1} - \frac{(1 + \eps)\E{X}}{\xi}
	~\text{ and }~
	f''(\xi) = -\frac{\mu^2 e^{\xi\mu}}{(e^{\xi\mu} - 1)^2} + \frac{(1 + \eps)\E{X}}{\xi^2}.
\]
The function $\frac{x}{(x-1)^2}$ is strictly monotone decreasing for $x > 1$. Together with the triangle inequality this implies for $\xi \ge 1$ 
\[
	|f''(\xi)| \le \frac{\mu^2 e^{\mu}}{(e^{\mu}-1)^2} + (1+\eps)\E{X} \le (\E{X} + 1 + \eps)\E{X}.
\]
Write $\xi = 1 + \delta$, where $\delta \ge 0$. We obtain by applying Taylor's Theorem
\[
	f(\xi) \le f(1) + f'(1)\delta + (\E{X} + 1 + \eps)\E{X}\delta^2.
\]
As $f(1) = 0$ and $f'(1) = -\eps \E{X}$, by setting $\delta = \frac{\eps}{2(1 + \eps + \E{X})}$ and using~\eqref{eq:f} we infer that 
\[
	\Pr{Y_{r'} \ge (1 + \eps)r'\E{X}} \le e^{-\frac{\eps^2}{4(1 + \eps + \E{X})}\E{Y_{r'}}}.
\]
We can deal completely analogously with the lower tail of the distribution of $Y_{r'}$ -- the straightforward details are omitted. 

Using the above bounds we obtain readily for any $0 \le r'\le r$ that there is a constant $C = C(\mu) > 0$  such that
\[
	\Pr{\left|Y_{r'} - r' \frac{\mu}{1 - e^{-\mu}}\right| \ge \eps r}
	\le e^{-C'\eps^2 r}.
\]
The claim for $j=1$ then follows with plenty of room to spare by applying the union bound and using that $\eps \ge \frac{\log r}{\sqrt{r}}$. As the calculations are very similar for the case $j = 2$ we leave them to the reader.
\end{proof}